\documentclass[11pt,a4paper,leqno]{article}
\usepackage[latin1]{inputenc}
\usepackage[T1]{fontenc}
\usepackage[english]{babel}
\usepackage{graphicx}
\usepackage{amsfonts,amssymb,amsbsy,latexsym,a4wide,xspace}
\usepackage{amsmath,delarray,enumerate,calc,amsthm}
\usepackage{bbm}
\usepackage{mathrsfs}
\usepackage{txfonts}
\usepackage{paralist}

\usepackage{authblk}

\usepackage{tikz}
\usetikzlibrary{matrix}

\usepackage{url}

\newtheorem{Th}{Theorem}[section]
\newtheorem{Prop}[Th]{Proposition}

\theoremstyle{definition}
\newtheorem{Remark}[Th]{Remark}

\newtheorem{Cor}[Th]{Corollary}

\newcommand{\supp}{\operatorname{supp}}

\newcommand{\cf}{{\mathcal F}}

\newcommand{\cD}{{\mathcal D}}

\newcommand{\cM}{{\mathcal M}}

\newcommand{\sB}{\mathscr{B}}

\newcommand{\N}{{\mathbbm N}}

\newcommand{\Z}{{\mathbbm Z}}
\newcommand{\A}{{\mathbbm A}}

\newcommand{\raz}{{\mathbbm 1}}

\newcommand{\lcm}{\operatorname{lcm}}

\newcommand{\beq}{\begin{equation}}

\newcommand{\eeq}{\end{equation}}

\renewcommand{\mod}{\text{ mod }}
\newtheorem {definition}{Definition}[section]
\newtheorem {lemma}{Lemma}[section]
\newtheorem{theorem}{Theorem}

\newtheorem {bemerkung}{Remark}[section]
\newtheorem{proposition}{Proposition}[section]
\newtheorem {corollary}{Corollary}[section]
\newtheorem{beispiel}{Example}[section]
\newtheorem{frage}{Question}[section]

\newenvironment{remark} {\begin{bemerkung} \normalfont }{\end{bemerkung}}
\newenvironment{example} {\begin{beispiel} \normalfont }{\end{beispiel}}

\author{Aurelia Dymek, Nazrul Haque and Stanis\l aw Kasjan}
\title{On automorphisms of $\mathscr{B}$-admissible and related subshifts}
\begin{document}
\maketitle
\begin{abstract}
    We adapt ideas of Kim and Roush \cite{KimRoush}, originally developed in the study of automorphisms of sofic subshifts, to obtain sufficient conditions under which a subshift has a huge automorphism group. We apply this approach to non-sofic subshifts defined by sets of multiples. In particular, we establish a dichotomy for the $\mathscr{B}$-admissible subshift: its automorphism group is either trivial or contains an embedded copy of the automorphism group of the full shift $\{0,1\}^{\mathbb Z}$. In the latter case, we say that the automorphism group is huge. We further show that the automorphism group of the hereditary closure of the $\mathscr{B}$-free subshift is huge whenever $\mathscr{B}\subset \mathbb N$ is infinite and contains no infinite pairwise coprime subset.
\end{abstract}

\section{Introduction}
\textbf{Subshifts} In this section we recall some concepts of the theory of symbolic dynamical systems, see \cite{LMa} for more information. We work over the smallest nontrivial alphabet, $\{0,1\}$. A subshift is a closed subset $X\subseteq \{0,1\}^{\Z}$ that is invariant under the left shift
\[
\sigma\colon \{0,1\}^{\Z}\to \{0,1\}^{\Z}, \qquad (\sigma x)[n]=x[n+1],
\]
where $x=(x[n])_{n\in\Z}$. For any $n,m\in\Z$ such that $n<m$ we will write $x[n,m]$ for $(x[k])_{k=n}^m$ and call it a \emph{block} or a \emph{word}. By the \emph{length} of the block $x[n,m]$ we will mean $m-n+1$ and denote it by $|x[n,m]|$. Every subshift can be described by a collection of forbidden blocks. An important class is formed by shifts of finite type (SFTs), for which this collection is finite. Taking factors of SFTs yields the class of sofic shifts.

In this paper we concentrate on hereditary closures of $\mathscr{B}$-free shifts and $\mathscr{B}$-admissible subshifts.
For any $\mathscr{B}\subset\N$, the set of $\mathscr{B}$-free numbers is the complement of the set $\cM_{\mathscr{B}}:=\bigcup_{b\in\mathscr{B}}b\Z$ of multiples of the elements of $\mathscr{B}$. The set of $\mathscr{B}$-free numbers is denoted by $\cf_{\mathscr{B}}$ and its characterisctic function by $\eta_{\mathscr{B}}$. The orbit closure of $\eta_{\mathscr{B}}$ is called the $\mathscr{B}$-free subshift and is denoted by $X_{\eta_{\mathscr{B}}}$.

A subshift $X\subset\{0,1\}^\Z$ is called \emph{hereditary} if \[x \in X \text{ and } z\preccurlyeq x \implies z \in X,\] where $``\preccurlyeq"$ means \textit{coordinatewise less than or equal to}.
For a subshift $X$, the \emph{hereditary closure} $\widetilde{X}$ of $X$ is the smallest hereditary subshift containing $X$. Equivalently, we have
\[\widetilde{X}=\{y\in\{0,1\}^\Z \colon \exists_{x\in X} \ y\preccurlyeq x\}.\]
It is known that the $\mathscr{B}$-free subshift is not necessarily hereditary, that is, $X_{\eta_{\mathscr{B}}}\varsubsetneq \widetilde{X_{\eta_{\mathscr{B}}}}$ in general, see \cite{DKKL}.

A set $A\subset\Z$ is called $\mathscr{B}$-admissible if, for every $b\in\mathscr{B}$, one has $|A \bmod b|<b$. If $\mathscr{B}=\{b\}$, we will write $b$-admissible instead of $\{b\}$-admissible. A sequence $x\in\{0,1\}^{\Z}$ is called $\mathscr{B}$-admissible if, its support, denoted by $\supp x$, is $\mathscr{B}$-admissible, where $\supp x=\{n\in\Z \ : \ x[n]\neq0\}$. The set of all $\mathscr{B}$-admissible sequences forms a subshift, called the \emph{$\mathscr{B}$-admissible subshift}, which we denote by $\A_{\mathscr{B}}$\footnote{The $\mathscr{B}$-admissible subshift is also denoted by $X_{\mathscr{B}}$, see \cite{DKKL}.}. If $\mathscr{B}=\emptyset$, then $\A_{\mathscr{B}}=\{0,1\}^{\Z}$ is the \emph{full shift}.

Since one may replace $\mathscr{B}$ by a set with the same set of multiples and assume it is primitive, that is $b\nmid b'$ for distinct $b,b'\in\mathscr{B}$, we henceforth assume that $\mathscr{B}$ is primitive. The study of sets of multiples and $\mathscr{B}$-free numbers has a long history, see \cite{MR1414678}. More recently, subshifts associated with $\mathscr{B}$-free numbers have been studied intensively; see, for example, \cite{DKKL,KKL} and the references therein. 

\textbf{Automorphism group} \emph{The automorphism group} $Aut(X)$ of a subshift $X$ consists of all homeomorphisms of $X$ that commute with the shift $\sigma$, with composition as the group operation. \footnote{The group $Aut(X)$ is the centralizer of $\sigma$ in the group of all homeomorphisms from $X$ to $X$.} It always contains the subgroup $\{\sigma^n\colon n\in\Z\}$. If this is the whole group, then $Aut(X)$ is called \emph{trivial}. By the Curtis--Hedlund--Lyndon Theorem \cite{Hed}, every automorphism is a block code, and hence $Aut(X)$ is countable.
By the {\em radius} of an automorphism $F\in Aut(\{0,1\}^{\Z})$ we mean the minimal number $N$ such that $F$ is defined by a block code $f\colon\{0,1\}^{[-N,N]}\rightarrow\{0,1\}$ \cite[Theorem 3.4]{Hed}. When we consider a map $F$, we usually denote by $f\colon \{0,1\}^{[-N,N]}\rightarrow\{0,1\}$ a block code defining $F$. Observe that if $N$ is the radius of an automorphism $F$ then the radius of the composition $\sigma^kF$ is less than or equal to $N+|k|$.

The following are our main results.  In Section \ref{sec:hereditary}, we show:
 \begin{theorem}\label{thm:main2}
 For any infinite $\mathscr{B}\subset\N\setminus\{1\}$ not containing an infinite pairwise coprime subset, $Aut(\{0,1\}^{\Z})$ embeds into $Aut(\widetilde{X_{\eta_{\mathscr{B}}}})$.
 \end{theorem}
 In \cite{Keller2} Keller proved that for any taut $\mathscr{B}$ the $\mathscr{B}$-free subshift is hereditary if and only if $\mathscr{B}$ contains an infinite pairwise coprime subset. Moreover, he showed that for any taut $\mathscr{B}$ containing an infinite pairwise coprime subset the automorphism group of $\mathscr{B}$-free subshift is trivial, see \cite{LRS}. Because of that Theorem \ref{thm:main2} provides a dichotomy:
 \begin{corollary}
For any taut $\mathscr{B}\subset\N\setminus\{1\}$ either $Aut(\{0,1\}^{\Z})$ embeds into $Aut(\widetilde{X_{\eta_{\mathscr{B}}}})$ or $Aut(\widetilde{X_{\eta_{\mathscr{B}}}})$ is trivial.
\end{corollary}
In the case of $\mathscr{B}$-admissible subshifts we deliver the following characterization of  non-triviality of $Aut(\A_{\mathscr{B}})$:
 \begin{theorem}\label{thm:main}
Assume that $\emptyset\neq \mathscr{B}\subset \N\setminus\{1\}$ is primitive. The following conditions are equivalent.
\begin{enumerate}
\item[(a)] There exists a nontrivial automorphism of $\A_{\mathscr{B}}$.
\item[(b)] The group of the automorphisms of the full shift $\{0,1\}^{\Z}$ embeds into that of  $\A_{\mathscr{B}}$.
\item[(c)] There exists a finite set $S\subseteq \mathscr{B}$ and an $S$-tuple  ${\bf r}=(r_b)_{b\in S}\in \Z^S$ of integers such that the set
 $$
K_{{\bf r}}:=\bigcup_{b\in S}(r_b+b\Z)
$$
is minimal  among all sets of the form $\bigcup_{b\in S}(r'_b+b\Z)$, where $(r'_b)_{b\in S}\in \Z^S$, and the set $\Z\setminus K_{\bf r}$ is
$\mathscr{B}$-admissible.\footnote{The set $K_{{\bf r}}$ satisfying the conditions in (c) is called an $S$-contour in  Definition \ref{def:contour}.}
\end{enumerate}
\end{theorem}
Theorem \ref{thm:main} allows us to obtain the corresponding dichotomy for automorphism group of $\A_\sB$:
 \begin{corollary}\label{thm:main1}
 For any $\mathscr{B}\subset\N\setminus\{1\}$ either $Aut(\{0,1\}^{\Z})$ embeds into $Aut(\A_{\mathscr{B}})$ or $Aut(\A_{\mathscr{B}})$ is trivial.
 \end{corollary}

 We provide also handy sufficient conditions for $\A_{\mathscr{B}}$ to have trivial automorphism group in Subsection \ref{sec:onlytrivial}.

 Let us recall that in \cite{Me}, Mentzen proved that the automorphism group of every Erd\H{o}s $\mathscr{B}$-free subshift is trivial, that is, whenever $\mathscr{B}$ is infinite, pairwise coprime, and satisfies $\sum_{b\in\mathscr{B}}1/b<\infty$.  This result was extended by Keller to taut sets $\mathscr{B}$ that contain an infinite pairwise coprime subset, see \cite{Keller2,LRS}. Theorem \ref{thm:main} can be seen as a generalization of that result. Note that $\widetilde{X_{\eta_{\mathscr{B}}}}=\A_{\mathscr{B}}$ in the Erd\H{o}s case \cite{DKKL}. This class of $\mathscr{B}$-free subshifts stands in sharp contrast to the class of $\mathscr{B}$-free Toeplitz shifts, for which $X_{\eta_\mathscr{B}}\neq\widetilde{X_{\eta_\mathscr{B}}}\neq\A_\mathscr{B}$, see \cite{KLA}. In the Toeplitz case there are examples with trivial automorphism group \cite{Dymek} as well as examples with nontrivial automorphism group \cite{DKK}.

 Results showing that the automorphism group of a subshift is large, are known mainly in the case of SFT or sofic subshifts, see e.g. \cite{BoLiRu}, \cite{KimRoush}, \cite{Sal}. The classes of subshifts considered in this paper is, in a sense, in sharp contrast to the class of sofic shifts; see Corollary~\ref{not_sofic} and Proposition \ref{lem:whensofic}.

However, to prove our results, we adapt the idea of Kim and Roush from \cite{KimRoush}. In this way, we obtain sufficient conditions on the language of a subshift $X$ that guarantee the existence of an embedding of $Aut(\{0,1\}^\Z)$ into $Aut(X)$, see Proposition~\ref{prop:hereditaryperiodic}. Then we verify that these conditions are satisfied for $\widetilde{X_{\eta_{\mathscr{B}}}}$ whenever $\mathscr{B}$ contains no infinite pairwise coprime subset, see Proposition~\ref{prop:taut_her}. To treat the case of $\A_\mathscr{B}$, we introduce certain unions of residue classes modulo elements of finite subsets of $\mathscr{B}$, called contours, see Definition~\ref{def:contour}. These contours constitute the main tool in the proof of Theorem~\ref{thm:main}, which is contained in Section \ref{sec:admissible}.
One of our conditions in Proposition~\ref{prop:hereditaryperiodic} is exchangeability of blocks. The notion of exchangeable blocks, to the best of our knowledge, does not exist in the literature.

Many results on automorphism groups of SFTs are obtained by manipulating of marker automorphisms, introduced in \cite{Hed} as a class of finite-order automorphisms of the full shift. Using marker automorphisms, Hedlund proved that the automorphism group of the full shift contains isomorphic copies of every finite group as well as the free group on two generators. In the same spirit, Kim and Roush \cite{KimRoush} showed that the automorphism group of the full shift embeds into the automorphism group of any mixing SFT, using markers to encode data words that serve as symbols of a full shift.

\section{A sufficient condition for huge automorphism group}
Given an automorphism $F\colon\{0,1\}^{\Z}\rightarrow \{0,1\}^{\Z}$ of the full shift, given by a block code $f:\{0,1\}^{[-N,N]}\rightarrow \{0,1\}$ we set $F^T:=\Delta\circ F\circ\Delta$, where $\Delta(x)[n]=x[-n]$ for every $x\in\{0,1\}^{\Z}$ and $n\in\Z$. Observe that $F^T$ is given by the block code $f^T$ defined by
$$
f^T(\alpha_{-N}\alpha_{-N+1}\ldots\alpha_{0}\ldots\alpha_{N-1}\alpha_{N})=f(\alpha_{N}\alpha_{N-1}\ldots\alpha_{0}\ldots\alpha_{-N+1}\alpha_{-N}).
$$
Moreover, $F^T$ is invertible with the inverse $\Delta\circ F^{-1}\circ\Delta$.

\begin{definition}\label{def:exchangeable}
Let $Y\subseteq\{0,1\}^{\Z}$ (not necessarily a subshift). We say that the blocks $C_1,C_2,\ldots$  of the same length $\ell$ are {\em mutually exchangeable in } $Y$, if for every $y\in Y$, if $y[n+1,n+\ell]=C_i$ for some $n\in\Z$ and $i$, then, for every $j$, we have $y'\in Y$, where $y'[k]=y[k]$ for $k\in\Z\setminus[n+1,n+\ell]$ and $y'[n+1,n+\ell]=C_{j}$.
\end{definition}
\begin{proposition}\label{prop:hereditaryperiodic}
Assume that $Y$ is a subshift in $\{0,1\}^{\Z}$ and $D_{0,0}$, $D_{0,1}$, $D_{1,0}$, $D_{1,1}$ $C$ are five blocks of the same length\footnote{We assume that  the length of $C$ equals the length of $D_{\alpha,\beta}$ only to simplify the technical aspects. Obviously, it is important that the blocks $D_{\alpha,\beta}$ have all equal lengths, as they are exchangeable, but $C$ can have arbitrary length.} $\ell$  appearing in $Y$ and satisfy the following conditions:
 \begin{enumerate}
 \item[(a)] $D_{0,0}$, $D_{0,1}$, $D_{1,0}$, $D_{1,1}$, $C$ are pairwise different.\label{conda}
 \item[(b)] The blocks $D_{0,0}$, $D_{0,1}$, $D_{1,0}$, $D_{1,1}$ are mutually exchangeable in $Y$.\label{condb}
 \item[(c)] Assume that $\varepsilon,\delta\in\{0,1\}$, $0\le m\le\ell$. If $C=(D_{\varepsilon,\delta}C)[m+1,m+\ell]$ then $m=\ell$ and if $C=(CD_{\varepsilon,\delta})[m+1,m+\ell]$ then $m=0$.\label{condc}
 \item[(d)] $Y$ contains the periodic element $\ldots CD_{0,0}CD_{0,0}\ldots CD_{0,0}C\ldots $.\label{condd}
 \end{enumerate}
 Then there exists an injective group homomorphism $\Phi:Aut(\{0,1\}^{\Z})\rightarrow Aut(Y)$.
\end{proposition}

\begin{proof}
The proof is essentially contained in \cite{KimRoush} and is based on the "bucket-passing" idea described there.
By an {\em active interval} of   $x\in Y$ we mean a block, appearing in $x$, being of the form $$CD_{\alpha_{0}, \beta_{0}}CD_{\alpha_{1}, \beta_{1}}C \ldots CD_{\alpha_{j}, \beta_{j}}C$$
or infinite to the right: $$CD_{\alpha_{0}, \beta_{0}}CD_{\alpha_{1}, \beta_{1}}C \ldots$$ or
interval as a block together with its position in $x$. We need a more precise definition of an active interval. For this purpose,  for $a\in\Z\cup\{-\infty\}, b\in \Z\cup\{+\infty\}$ and a sequence $x\in\{0,1\}^\Z$ let us denote by $x(a,b)$  the sequence of the terms of $x$ that are indexed by the integers $i$ satisfying $a<i<b$. An active interval of $x$ is a triple $(a,b,B)$, where $a\in\Z\cup\{-\infty\}, b\in \Z\cup\{+\infty\}$, $B\in\{0,1\}^{\Z}$ and:
\begin{itemize}
\item the sequence $B$ is of the form\footnote{By the exchangeability condition (b) and the fact that $Y$ is closed, it follows from (d) that every such sequence belongs to $Y$.}\begin{equation}\label{eq:sequence}
\ldots CD_{\alpha_{-2}, \beta_{-2}}CD_{\alpha_{-1}, \beta_{-1}}CD_{\alpha_{0}, \beta_{0}}CD_{\alpha_{1}, \beta_{1}}CD_{\alpha_{2}, \beta_{2}}C\ldots,
\end{equation}
where $\alpha,\beta\in\{0,1\}^{\Z}$.
\item $x(a,b)=B(a,b)$.
\item If $b\in \Z$ and $a=-\infty$, then
$$
B(a,b)=\ldots CD_{\alpha_{-2}, \beta_{-2}}CD_{\alpha_{-1}, \beta_{-1}}CD_{\alpha_{0}, \beta_{0}}C
$$
and $(\alpha_1,\alpha_2,\ldots)=(\beta_0,\beta_{-1},\beta_{-2},\ldots)$, $(\beta_1,\beta_2,\ldots)=(\alpha_0,\alpha_{-1},\alpha_{-2},\ldots)$. In this case the sequences $\alpha=\alpha(a,b,B)$ and $\beta=\beta(a,b,B)$ are defined uniquely by $(a,b,B)$ (we understand that $\alpha_0$ and $\beta_0$ are the $0$th positions) and $\sigma(\beta(a,b,B))=\Delta(\alpha(a,b,B))$.
\item If $a\in \Z$ and $b=+\infty$, then
$$
B(a,b)=CD_{\alpha_{0}, \beta_{0}}CD_{\alpha_{1}, \beta_{1}}CD_{\alpha_{2}, \beta_{2}}C\ldots
$$
and $(\ldots,\alpha_{-2},\alpha_{-1})=(\ldots,\beta_1,\beta_0)$, $(\ldots,\beta_{-2},\beta_{-1})=(\ldots,\alpha_1,\alpha_0)$. Again, the sequences $\alpha=\alpha(a,b,B)$ and $\beta=\beta(a,b,B)$ are defined uniquely by $(a,b,B)$  and $\sigma^{-1}(\beta(a,b,B))=\Delta(\alpha(a,b,B))$.
\item If $a,b\in \Z$, then
$$
B(a,b)=CD_{\alpha_{0}, \beta_{0}}CD_{\alpha_{1}, \beta_{1}}CD_{\alpha_{2}, \beta_{2}}C\ldots D_{\alpha_{r}, \beta_{r}}C,
$$
where $(2r+3)\ell=b-a-1$ and $\alpha=\alpha(a,b,B)$ is the periodic sequence
$$
\ldots \underline{\alpha_0}\alpha_1\ldots\alpha_r\beta_r\beta_{r-1}\ldots\beta_1\beta_0\alpha_0\alpha_1\ldots.
$$
  and $\beta(a,b,B)=\sigma\circ\Delta(\alpha(a,b,B))$. The underlined character indicates the 0th position in the sequence.
\end{itemize}
The sequences $\alpha$, $\beta$ carry information on the block $x(a,b)=B(a,b)$; the part of $B$ outside the interval $(a,b)$ is invisible in $x$, however it is convenient (although a little counter-intuitive) to keep the sequences $\alpha, \beta$ infinite in all cases, that is why we "roll them up" in the cases $(a,b)\neq (-\infty,+\infty)$.

We shall use the concept of a maximal active interval. The maximality of an active interval is understood in an intuitive way. In precise terms,  an active interval $(a,b,B)$ of $x$ is maximal, if  for every active interval $(a',b',B')$ of $x$ such that $a'\le a$ and $b\le b'$, we have $(a',b',B')=(a,b,B)$.

Let $F$ be an automorphism of $\{0,1\}^{\Z}$.  We associate to $F$ a map $F'$ on $Y$, which maps a sequence of the form (\ref{eq:sequence})
 (i.e. a sequence being a two-sided infinite active interval), to the sequence
$$
\ldots CD_{F(\alpha)_{-2}, F^T(\beta)_{-2}}CD_{F(\alpha)_{-1}, F^T(\beta)_{-1}}CD_{F(\alpha)_{0}, F^T(\beta)_{0}}CD_{F(\alpha)_{1}, F^T(\beta)_{1}}CD_{F(\alpha)_{2}, F^T(\beta)_{2}}C\ldots.
$$
Then we use "bucket-passing construction" (see \cite{KimRoush}) to extend $F'$ on the whole $Y$.

Every maximal active interval $(a,b,B)$ with $(a,b)\neq(-\infty,+\infty)$ of $x$ is replaced by an active interval $(a,b,B')$, where
$$
\alpha(a,b,B')=F(\alpha(a,b,B))\;\text{and}\; \beta(a,b,B')=F^T(\beta(a,b,B)).$$

For example, a maximal active interval
$$CD_{\alpha_{0}, \beta_{0}}CD_{\alpha_{1}, \beta_{1}}C \ldots CD_{\alpha_{j}, \beta_{j}}C$$ is replaced by
$$CD_{\alpha'_{0}, \beta'_{0}}CD_{\alpha'_{1}, \beta'_{1}}C \ldots CD_{\alpha'_{j}, \beta'_{j}}C,$$
where
$$
\ldots \beta'_0\underline{\alpha'_0}\alpha'_1\ldots\alpha'_j\beta'_j\beta'_{j-1}\ldots\beta'_0\alpha'_0\alpha'_1\ldots=F(\ldots \beta_0\underline{\alpha_0}\alpha_1\ldots\alpha_j\beta_j\beta_{j-1}\ldots\beta_0\alpha_0\alpha_1\ldots).
$$
Observe that the l.h.s sequence is periodic, as the sequence at the r.h.s is, with the same period.


We claim that the sequence $F'(y)$ is well defined for every $y\in Y$. Observe first that it follows by the conditions (a) and (c) that if two maximal active intervals of $x\in Y$ overlap along a block of length at least $\ell$, then they must be equal, as their union is again an active interval.\footnote{In other words,
if two maximal active intervals of $x\in Y$ overlap along a block $C'$ then $|C'|<|C|$, comp. the overlap condition in section 2.2 of \cite{Ya}.}
The above construction may change the $n$th ($n\in\Z$) term of $y$ if and only if $n$ belongs to some block of type $D_{\varepsilon,\delta}$ contained in a maximal active interval of $y$. Thanks to the previous remark, this maximal active interval is uniquely determined by $n$ (and $y$), so the claim follows.

It is clear from the construction that $F'$ commutes with the shift $\sigma$.
The continuity of $F'$ follows since the value of the $n$-the term of $F'(x)$ depends only on the interval $x[n-(2N+3)\ell,n+(2N+3)\ell]$, where $N$ is a radius of $F$.
In order to prove that,
assume that
\begin{equation}\label{eq:xiy}
x[n-(2N+3)\ell,n+(2N+3)\ell]=y[n-(2N+3)\ell,n+(2N+3)\ell]
\end{equation} for some $x,y\in Y$.
We claim that $F'(x)[n]=F'(y)[n]$. By saying that $n$ {\em belongs to a maximal active interval of} $x$  we shall mean that there exists a maximal active interval $(a,b,B)$ of $x$ such that $a<n<b$. If this is the case then, by (\ref{eq:xiy}) we conclude that
$$
x[m+1,m+3\ell]=CD_{\varepsilon,\delta}C= y[m+1,m+3\ell]
$$
for some $\varepsilon,\delta\in\{0,1\}$ and $m$ satisfies $m+1\le n\le m+3\ell$. It follows that $n$ belongs to a maximal active interval of $y$. The same is true with the roles of $x$ and $y$ interchanged. It follows that if  $n$ does not belong to any active interval of $x$ then $F'(x)[n]=x[n]=y[n]=F'(y)[n]$. So assume otherwise, let $(a,b,B)$ and $(a',b',B')$ be maximal active intervals of $x$ and $y$, respectively, such that $a<n<b$ and $a'<n<b'$. Note that, due  to (\ref{eq:xiy}), if $b\le n+(2N+1)\ell+1$, then $b=b'$. Similarly, if $a\ge n-(2N+1)\ell-1$ then $a=a'$. There are several cases to consider.

1) $b=b'\le n+(2N+1)\ell+1$ and $a=a'\ge n-(2N+1)\ell-1$. Then by (\ref{eq:xiy}), we have $B=B'$ and $F'(x)[n]=F'(y)[n]$.

2) $b=b'\le n+(2N+1)\ell+1$ and $a,a'< n-(2N+1)\ell-1$. In this case, again by (\ref{eq:xiy}), the restrictions of $B$ and $B'$ to the interval $[n-(2N+1)\ell,n+(2N+1)\ell]$. If $n$ "lies under $C$", that is $b-(2r+1)\ell\le n <\ell-2r\ell$ for some $r\in\N_0$, then $F'(x)[n]=x[n]=y[n]=F'(y)[n]$. Otherwise, when $b-2r\ell\le n<b-(2r-1)\ell$ for some $r\in\N$ ($r\le N$ in this case), the restrictions of $B$ and $B'$ to the interval $[n-(2N+1)\ell,n+(2N+1)\ell]$ are equal to
$$
U'D_{\alpha_{-r-N},\beta_{-r-N}}CD_{\alpha_{-r-N+1},\beta_{-r-N+1}}C\ldots CD_{\alpha_{-r},\beta_{-r}}C\ldots CD_{\alpha_{0},\beta_{0}}CU''
$$
for some blocks $U',U''$. Then
 we have $F'(x)[n]=F'(y)[n]$, as both values depend only on the sequence
$$
\alpha_{-r-N},\alpha_{-r-N+1},\ldots,\alpha_{-r},\ldots,\alpha_0,\beta_0,\beta_{-1},\ldots,\beta_{r+1-N}.
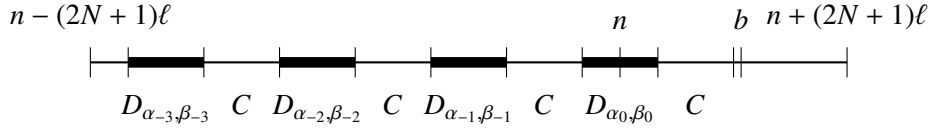
$$

\bigskip

\begin{tikzpicture}
\draw[thick] (0,0) -- (10,0);

\draw[line width=4pt] (0.5,0) -- (1.5,0);
\draw[line width=4pt] (2.5,0) -- (3.5,0);
\draw[line width=4pt] (4.5,0) -- (5.5,0);
\draw[line width=4pt] (6.5,0) -- (7.5,0);

\foreach \x in {0,0.5,1.5,2.5,3.5,4.5,5.5,6.5,7,7.5,8.5,8.6,10}
  \draw (\x,0.2) -- (\x,-0.2);
\node[below] at (1,-0.3) {$D_{\alpha_{-3},\beta_{-3}}$};
\node[below] at (3,-0.3) {$D_{\alpha_{-2},\beta_{-2}}$};
\node[below] at (5,-0.3) {$D_{\alpha_{-1},\beta_{-1}}$};
\node[below] at (7,-0.3) {$D_{\alpha_0,\beta_0}$};
\node[below] at (2,-0.3) {$C$};
\node[below] at (4,-0.3) {$C$};
\node[below] at (6,-0.3) {$C$};
\node[below] at (8,-0.3) {$C$};
\node[above] at (0,0.3) {$n-(2N+1)\ell$};
\node[above] at (10,0.3) {$n+(2N+1)\ell$};
\node[above] at (8.6,0.3) {$b$};

\node[above] at (7,0.3) {$n$};

\end{tikzpicture}

\begin{center}
The case $N=2$, $r=1$.
\end{center}

\bigskip

3) The case $b,b'> n+(2N+1)\ell+1$ and $a=a'\ge n-(2N+1)\ell-1$ is  treated as case 2).

4) $b,b'> n+(2N+1)\ell+1$ and $a,a'< n-(2N+1)\ell-1$. In this case, either $n$ lies under $C$, and then $F'(x)[n]=x[n]=y[n]=F'(y)[n]$ or the restrictions of $B$, $B'$ to the interval $[n-(2N+1)\ell,n+(2N+1)\ell]$ coincide with the block of the form
$$
U'D_{\alpha_{-N},\beta_{-N}}CD_{\alpha_{-N+1},\beta_{-N+1}}C\ldots CD_{\alpha_{0},\beta_{0}}C\ldots CD_{\alpha_{N},\beta_{N}}CU''
$$
for some blocks $U',U''$. In this case $F'(x)[n]=F'(y)[n]$ as both values depend only on the sequences
$$
\alpha_{-N},\ldots,\alpha_0,\ldots,\alpha_N\; \text{ and }\; \beta_N\ldots,\beta_0\ldots \beta_{-N}.
$$
The claim follows.

By the condition (b), $F'(y)\in Y$ for every $y\in Y$.

It follows easily from the definition that the map $F\mapsto F'$ is a group homomorphism.
We finish the proof by showing that this is injective. Assume that $F$ is not the identity and let ${\alpha}\in\{0,1\}^{\Z}$ be such that $F({\alpha})\neq {\alpha}$. Then $F'(y)\neq y$, where $y$ is the sequence (\ref{eq:sequence}) with $\beta$ arbitrary. It remains to notice that $y\in Y$ thanks to the conditions (d) and (b).
\end{proof}
\begin{remark}
The condition (c) in Proposition \ref{prop:hereditaryperiodic} is slightly weaker than the condition that $C$ does not overlap with itself nontrivially. However, in what follows we  apply the proposition with $C$ satisfying the latter condition.
If $Y$ is an uncountable sofic subshift, then $C, D_{0,0}, D_{0,1}, D_{1,0}, D_{1,1}$ satisfying the conditions (a)-(d) can be constructed with the help of synchronizing words. This seems to be rather folklore, compare Lemma 1 in \cite{Sal}.
\end{remark}
\section{On automorphisms of the hereditary closure of the $\mathscr{B}$-free shift}\label{sec:hereditary}

The classical \textit{Dirichlet theorem on primes} in arithmetic progressions has the following corollary:


\begin{Cor}\label{Cor_Dirichlet}
     Let $\cD=\{d_1, d_2, \ldots, d_{\ell}\}\subseteq \mathbb{N}\setminus \{1\}$ be any finite primitive set. Then for any $r\in \mathbb{N}$
     \begin{equation}\label{eq:e-1}
         r+d_i \Z \nsubseteq \bigcup_{k\neq i} d_k \Z.
     \end{equation}
\end{Cor}

\begin{proof}
Let $\gcd(d_i, r)=g$. Then, we can write $r+ d_i \Z = g\cdot (r' + d_i^{'} \Z)$ where $r' = \frac{r}{g}$, $d_{i}^{'} = \frac{d_i}{g}$ and $\gcd (d_{i}^{'}, r')=1$. By 
Dirichlet theorem, $r'+d_{i}^{'}\Z$ contains infinitely many primes. Let us choose a prime $p >\max (\cD)$ so that $p \in r'+d_{i}^{'}\Z$. Observe that $x= g\cdot p \in r+d_i \Z$.

Assume for a contradiction that $r+ d_i \Z \subseteq \bigcup_{k\neq i} d_k \Z$. Then there exists $d_k \in \cD$ with $k\neq i$ such that $d_k \mid x=g \cdot p$. This implies that $d_k \mid \gcd(d_i, r)$ as $\gcd(d_k, p) = 1$, due to the choice of $p>\max (\cD)$. Now,  $d_k \mid d_i$, which is a contradiction, since $\cD$ is primitive. Hence, the claim follows.
\end{proof}

The above statement generalizes to the case of taut $\cD$, see Proposition 4.31 in \cite{DKKL}.

\begin{Prop}(\cite[Corollary 4.32]{DKKL})\label{pro1}
    Let $\cD\subseteq \mathbb{N}\setminus \{1\}$ be any finite primitive set. If $L= \lcm (\cD) $, then for all $L'\geq L$ and $d\in \cD$:
    \begin{equation*}
        \left|(\cf_{\cD} \cap [1, L'])\mod d\right| = d-1.
    \end{equation*}
\end{Prop}

\begin{proof}
Let $\cD= \{d_1, d_2, \ldots, d_m\}\subseteq \mathbb{N}\setminus\{1\}$ be primitive and fix $d_i \in D$. First note that $d_i\Z\cap\cf_{\cD}=\emptyset$, so , $|\cf_{\cD} \mod d_i|\leq d_i-1$ for all $d_i\in \cD$. Thanks to the primitivity of $\cD$, we have (\ref{eq:e-1}). Choose $r\in \Z$ such that $d_i \nmid r$. As $(r+d_i\Z)\cap d_i\Z=\emptyset$, it follows by (\ref{eq:e-1}) that $(r+d_i\Z)\cap \cf_{\cD}\neq\emptyset$. Hence
  \begin{equation}\label{eq:e01}
      |\cf_{D} \mod d_i| = d_i-1
  \end{equation}
for every $d_i\in D$. Finally, let $L= \lcm (\cD)$ and note that $\cf_{\cD}$ is $L$-periodic. Choose any $d\in \cD$ and $x\in \cf_{\cD}$. Then there exists $y\in \{1, 2, \ldots, L\}$ such that $x=L\cdot q +y$ for $q \in \mathbb{N}$. Thanks to the periodicity of $\cf_{\cD}$, we have $x\in \cf_{\cD}$ if and only if $y\in \cf_{\cD}$. Thus, for any $d_i \in {\cD}$, every non-zero residue modulo $d_i$ is represented in the set $\cf_{\cD} \cap [1, L]$. Hence the claim follows easily for any $L'\geq L$.
\end{proof}

\begin{Prop}\label{pro2}
    Assume that $\mathscr{B}\subseteq\mathcal{M}_{\mathscr{Q}}$, where $\mathscr{Q}\subset \N\setminus\{1\}$ is a finite primitive set. Then there exists a finite primitive set $\mathscr{Q}'\subseteq \mathbb{N}\setminus \{1\}$ such that $\mathscr{B}\subseteq \mathcal{M}_{\mathscr{Q}'}$ and \begin{equation}\label{eq:e1}
\forall_{q\in\mathscr{Q}'}\exists_{b\in\mathscr{B}}q=\gcd(b,\lcm(\mathscr{Q}')).
\end{equation}
\end{Prop}
\begin{proof}
    Let us define $\mathscr{Q}' = \{\gcd(b, \lcm(\mathscr{Q}))~:~b\in \mathscr{B} \}$. Take any $b \in \mathscr{B}$. Then $ \gcd(b, \lcm(\mathscr{Q})) \mid b$ and this implies that $\mathscr{B} \subseteq \mathcal{M}_{\mathscr{Q}'}$. Since $\mathscr{B}\subseteq \mathcal{M}_{\mathscr{Q}}$, for every $b\in \mathscr{B}$ there exists at least one $q\in \mathscr{Q}$ such that $q\mid b$. Also, $q\mid \lcm({\mathscr{Q})}$ for every $q\in \mathscr{Q}$. Thus, $q\mid \gcd (b, \lcm(\mathscr{Q}))$. This implies that $\gcd (b, \lcm(\mathscr{Q})) >1$, in other words, $\mathscr{Q}' \subseteq \mathbb{N}\setminus \{1\}$.

    Observe that $\mathscr{Q}'$ is finite. It remains to prove (\ref{eq:e1}). For every $q\in \mathscr{Q}'$, there is $b_q\in \mathscr{B}$ such that $\gcd (b_q, \lcm (\mathscr{Q})) = q$ by the definition of $\mathscr{Q}'$. Then $q \mid \gcd(b_q, \lcm(\mathscr{Q}'))$. Now, $\lcm (\mathscr{Q}') \mid \lcm {(\mathscr{Q})}$ and therefore $ \gcd (b_q, \lcm (\mathscr{Q}')) \mid \gcd (b_q, \lcm (\mathscr{Q})) = q$. So, $\gcd (b_q, \lcm (\mathscr{Q}')) = q$ for any $q\in \mathscr{Q}'$.

\end{proof}

\begin{Remark}\label{rmk1}
    By the above proposition, if $\mathscr{B} \subseteq \mathcal{M}_{\mathscr{Q}}$, then we can always assume that there exists a finite primitive set $\mathscr{Q}\subseteq \mathbb{N}\setminus \{1\}$ that satisfies the condition (\ref{eq:e1}). Note that $\mathscr{Q}$ with these properties is not uniquely determined by $\mathscr{B}$.
\end{Remark}

\begin{lemma}\label{le1}
    Assume that $\mathscr{B}\subseteq \mathcal{M}_{\mathscr{Q}}$, where $\mathscr{Q}\subset \N\setminus\{1\}$ is a finite primitive set having the property (\ref{eq:e1}). For each $q\in \mathscr{Q}$, choose $b_q \in \mathscr{B}$ such that $\gcd (b_q, \lcm(\mathscr{Q}))=q$  and define \begin{align}\label{eq:e2}
    K= \max_{q\in\mathscr{Q}} \frac{b_q}{q},~~M=(K+1)\lcm (\mathscr{Q}).
\end{align}
 Then for every $M'\geq M$ and every integer $n$ the following condition holds: if $\eta_{\mathscr{Q}}[1,M']\preccurlyeq \eta_{\mathscr{B}}[n+1,n+M']$, then $\eta_{\mathscr{Q}}[1,M']\preccurlyeq \eta_{\mathscr{Q}}[n+1,n+M']$ and $\lcm(\mathscr{Q})|n$.
\end{lemma}

\begin{proof}

Fix $M'\geq M$ and $n\in \Z$ and assume that
\begin{equation}\label{eq:assumption}
\eta_{\mathscr{Q}}[1,M']\preccurlyeq \eta_{\mathscr{B}}[n+1,n+M'].
\end{equation}
  Take any $i \in \{1, 2, \ldots, \lcm(\mathscr{Q})\}$ such that $\eta_{\mathscr{Q}} [i] = 1$. Then, $\eta_{\mathscr{B}} [i+n] = 1$.  Assume that $\eta_{\mathscr{Q}}[i+n] = 0$. Then there exists $q\in\mathscr{Q}$ such that $q \mid i+n$, i.e. $\gcd(b_q, \lcm(\mathscr{Q})) \mid i+n$. Hence, the linear Diophantine equation
  \begin{align}
      b_q x - \lcm(\mathscr{Q})y = i+n
  \end{align}
has an integer solution $(x, y)$. Moreover, it  has a unique solution $(x, y)$ with  $ 0\leq y<\frac{b_q}{q}$. Define $a=i+y\lcm(\mathscr{Q})=b_qx - n$.
Since $\eta_{\mathscr Q}$ is $\lcm(\mathscr{Q})$-periodic and $i\in \mathcal{F}_{\mathscr{Q}}$,
\[
a\in i+\lcm(\mathscr{Q})\Z\subset\mathcal F_{\mathscr Q}.
\]
Moreover,
\[
1\le a< i+K\lcm(\mathscr{Q})\leq M\leq M'.
\]
Thus $a\in\mathcal F_{\mathscr Q}\cap[1,M']$, which, in view of the condition (\ref{eq:assumption}), implies
\[
a+n\in\mathcal F_{\mathscr B}\cap[n+1,n+M'].
\]
On the other hand, by construction $a+n=b_s x\in b_s\Z$, hence $a+n\notin\mathcal F_{\mathscr B}$.
This is a contradiction. This proves $\eta_{\mathscr{Q}}[i+n] = 1$ whenever $\eta_{\mathscr{Q}}[i] = 1$. Since $i$ is arbitrary and $\eta_{\mathscr{Q}}$ is $\lcm(\mathscr{Q})$-periodic, we conclude that $\eta_{\mathscr{Q}}[1,M']\preccurlyeq \eta_{\mathscr{Q}}[n+1,n+M'].$
It follows \begin{align}\label{eq:e4}
    \mathcal{F}_{\mathscr{Q}} \cap [1, M'] \subset \mathcal{F}_{\mathscr{Q}} \cap [n+1, n+M'] - n.
\end{align}
Thanks to the fact that $\lcm(\mathscr{Q}) < M$, by Proposition \ref{pro1}, we have
$\left|\supp~\eta_{\mathscr{Q}}[1, M']\mod q\right|=q-1$ for all $q\in \mathscr{Q}$. In other words, for any given integer $m$,
\begin{align}\label{eq:e5}
    \mathcal{F}_{\mathscr{Q}} \cap [1, M'] \cap (m+q\Z)=\emptyset~~\text{implies that}~q\mid m
\end{align}

 For every $q\in \mathscr{Q}$, we have $(q\Z - n) \cap \left[(\mathcal{F}_{\mathscr{Q}}\cap [n+1, n+M'])-n\right] = \emptyset.$ Therefore, from $(\ref{eq:e4})$, it follows that
 $(q\Z - n) \cap (\mathcal{F}_{\mathscr{Q}}\cap [1, M'])= \emptyset.$ From (\ref{eq:e5}), we have $q\mid n$ and this is true for any arbitrary $q\in \mathscr{Q}$. Hence, $\lcm (\mathscr{Q})\mid n$.

\end{proof}

We will use the following fact.

\begin{Prop}(\cite[Theorem 3.7]{DKKL})
A set $\mathscr{B}\subseteq \mathbb{N}\setminus\{1\}$ contains an infinite pairwise coprime subset if and only if $\mathscr{B} \nsubseteq \mathcal{M}_{\mathscr{Q}}$ for any finite set $\mathscr{Q}\subseteq \mathbb{N}\setminus \{1\}$.
\end{Prop}
As we assumed that $\mathscr{B}$ does not contain any infinite pairwise coprime set, therefore by the above
theorem, there exists $\mathscr{Q}\subseteq \mathbb{N}\setminus \{1\}$ such that
\begin{align}\label{aa}
\mathscr{B} \subseteq \mathcal{M}_{\mathscr{Q}}
\end{align}
Thanks to Remark \ref{rmk1}, we can assume that $\mathscr{Q}$ satisfies the condition (\ref{eq:e1}).

\begin{Remark}\label{rmk2}
    For $\mathscr{Q}\subseteq \mathbb{N}\setminus \{1\}$, if $k \in \lcm (\mathscr{Q}) \Z$, then $k-1, k+1 \in \mathcal{F}_{\mathscr{Q}}$.
\end{Remark}
Let $M:= (K+1)\lcm(\mathscr{Q})$  where $K$ is defined in (\ref{eq:e2}) above. Thanks to Proposition \ref{pro1}, we have  $\mid\supp ~\eta_{\mathscr{Q}}[1, M]\mod q\mid = q-1$ for all $q\in \mathscr{Q}$. Define a finite block
\begin{equation}\label{eq:B}
    B:= \eta_{\mathscr{Q}} [1, M].
\end{equation} 
By the construction of $B$, we have $B[1]=1$ and $B[M]=0$ . As $\lcm(\mathscr{Q})\mid M$, hence by Remark \ref{rmk2}, $M-1\notin \mathcal{M}_{\mathscr{Q}}$. So, we get $B[M-1]=1$. Hence, each such block $B$ has at least two non-zero entries.

\begin{Prop}\label{pro4}
    If $\eta_{\mathscr{Q}}[1, M'] \preccurlyeq \eta_{\mathscr{B}}[s+1, s+M']$ for some $s\in \mathbb{N}$ and $M'\geq M$ such that $\lcm (\mathscr{Q}) \mid M'$, then $\eta_{\mathscr{B}}[s+M'+1]=1=\eta_{\mathscr{B}}[s+2M'-1].$
\end{Prop}

\begin{proof}
    Using Lemma \ref{le1}, we obtain $\lcm(\mathscr{Q})\mid s$. Moreover, $\lcm (\mathscr{Q})\mid s+M'$  and $\lcm (\mathscr{Q})\mid s+2M'$ (as $\lcm (\mathscr{Q})\mid M'$). By Remark \ref{rmk2}, we obtain $s+M'+1, s+2M'-1 \notin \mathcal{M}_{\mathscr{B}}$. Thus, it follows that $\eta_{\mathscr{B}}[s+M'+1]=1=\eta_{\mathscr{B}}[s+2M'-1]$.
\end{proof}

\begin{Remark}\label{rmk3} Remark 5.4 in \cite{MR4938646} says \[\widetilde{X_{\eta_{\mathscr{B}}}} =\overline{\{\sigma^ny \colon y\preccurlyeq x,n\in\Z\}}.\] So
  for any $z\in \{0, 1\}^{\Z}$ the following conditions are equivalent:
  \begin{itemize}
      \item[(a)] $z\in \widetilde{X_{\eta_{\mathscr{B}}}}$.

      \item[(b)] there exists a sequence $(q_m)_{m\in\N}$  of integers and a sequence $(m_k)_{k\in\N}$ such that for every  $k\in \mathbb{N}$ large enough and all  $m>m_k$: $$z[-k, k] \preccurlyeq \sigma^{q_m} ({\eta_{\mathscr{B}}})[-k, k].$$
  \end{itemize}
\end{Remark}

\begin{Prop}\label{prop:taut_her}
    Assume that $\mathscr{B}\subseteq \mathcal{M}_{\mathscr{Q}}$ for a  finite primitive set $\mathscr{Q}\subseteq \mathbb{N}\setminus \{1\}$. Then there exist blocks $D_{\varepsilon, \delta}$ for $\varepsilon, \delta \in \{0, 1\}$ and $ C$ satisfying the hypothesis of Proposition \ref{prop:hereditaryperiodic} with $Y= \widetilde{X_{\eta_{\mathscr{B}}}}$.
\end{Prop}

\begin{proof} Thanks to Proposition \ref{pro2} we can assume that $\mathscr{Q}$ has the property (\ref{eq:e1}).
    Take $\ell = 3M$ and  set the words \begin{align*}
       D_{\varepsilon, \delta} &:= B\varepsilon 0^{\frac{\ell}{3}-3}\delta 0 B,\\
        C&:= 10^{\ell -1}.
    \end{align*}
    where $\varepsilon, \delta \in \{0, 1\}$ and $B$ is as in $(\ref{eq:B})$. Due to the periodicity of $\eta_{\mathscr{Q}}$ and by Proposition \ref{pro4}, it  follows that
    \begin{equation}\label{eq:e7}
        D_{\varepsilon, \delta}\preccurlyeq \eta_{\mathscr{Q}}[1, \ell] \preccurlyeq \eta_{\mathscr{B}}[1, \ell].
    \end{equation}

     We verify that the blocks $D_{\varepsilon, \delta}$ and $ C$ satisfy the hypothesis of Proposition \ref{prop:hereditaryperiodic}.

\begin{itemize}
    \item[(a)] This is immediate due to the fact that $ D_{\varepsilon, \delta} [M+1] =\varepsilon, D_{\varepsilon, \delta}[2M-1] = \delta, D_{\varepsilon, \delta}[2M+1]=1$  and $C[2M+1] = 0$.

    \item[(b)] Assume that $z \in \widetilde{X_{\eta_{\mathscr{B}}}}$ is such that $z[r+1, r+\ell] = D_{\varepsilon, \delta}$ for some $r\in \Z$ and $\varepsilon, \delta \in \{0, 1\}$. Take any sequence $z' $ such that $z' [r+1, r+\ell] = D_{\varepsilon', \delta'}$, where $\varepsilon', \delta' \in \{0, 1\}$ and $z'[i] =z[i]$ for $i\in \Z\setminus [r+1, r+\ell]$. We claim $z' \in \widetilde{X_{\eta_{\mathscr{B}}}}$.

As $z \in \widetilde{X_{\eta_{\mathscr{B}}}}$,  by Remark \ref{rmk3},
there exists a sequence $(q_m)_{m\in\N}$  of integers and a sequence $(m_k)_{k\in\N}$ such that for every  $k\in \mathbb{N}$ large enough and all  $m>m_k$:
    \begin{equation}\label{eq:e8}
        z[-k, k] \prec \sigma^{q_m} (\eta_{\mathscr{B}})[-k, k].
    \end{equation}

   Fix $k\geq r+\ell$, so that $z[1, k]$ contains the block $D_{\varepsilon, \delta}$. Due to our specific construction of $D_{\varepsilon, \delta}$, we see that $z'$ can differ from $z$ at most at two positions, namely at $r+M+1$ and $r+2M-1$. In other words, for all  $m>m_{k}$ and all $i \in \Z\setminus \{r+M+1, r+2M-1\}$ with $|i|\leq k$,
    \begin{equation}\label{eq:e9}
        z'[i]= z[i] \leq \sigma^{q_m} (\eta_{\mathscr{B}})[i].
    \end{equation}

   So, to prove our claim, we just need to show that $z'[r+M+1] \leq \sigma^{q_m} \eta_{\mathscr{B}} [r+M+1]$ and $z'[r+2M-1] \leq \sigma^{q_m} \eta_{\mathscr{B}}[r+2M-1]$.

Observe that $z[r+1, r+M] = z[r+2M+1, r+\ell] = B$.  It follows that $$z[r+1, r+M]=B= \eta_{\mathscr{Q}}[1, M]\preccurlyeq \eta_{\mathscr{B}} [q_m+r+1, q_m+r+M].$$ Thanks to Proposition \ref{pro4}, we have $\sigma^{q_m} \eta_{\mathscr{B}} [r+M+1] = 1 = \sigma^{q_m} \eta_{\mathscr{B}} [r+2M-1]$. Hence, $z'[r+M+1] =\varepsilon'\leq \sigma^{q_m} \eta_{\mathscr{B}}[r+M+1]=1$ and $z'[r+2M-1]=\delta' \leq \sigma^{q_m} \eta_{\mathscr{B}}[r+2M-1]=1$. So, from (\ref{eq:e9}), the inequality $z'[i] \leq \sigma^{q_m}\eta_{\mathscr{B}}[i]$ holds for all $i\in \Z$ with $|i|\leq k$. Therefore, by Remark \ref{rmk3}, the claim follows.

\item[(c)] Since $C$ starts with $1$, it follows directly from the construction.

\item[(d)] This follows directly from the fact that $\eta_{\mathscr{Q}} \in \widetilde{X_{\eta_{\mathscr{B}}}}$ and by (\ref{eq:e7}), every sequence described in $(d)$ is coordinatewise less than or equal to $\sigma^m (\eta_{\mathscr{Q}})$ for some $m \in \Z$.
\end{itemize}

\end{proof}

Proposition \ref{prop:taut_her} together with Proposition \ref{prop:hereditaryperiodic} prove Theorem \ref{thm:main2}.

\begin{example}
    To understand the above construction, let's look at the following simple example. Take $$\mathscr{Q} = \{4, 6, 10\}$$ and set $\mathscr{B} = \{8c_i : i\in\N\} \cup \{18 c_i : i\in\N\}$ where $\{c_i : i\in\N\}$ are pairwise coprime such that $c_1 = 1, \gcd(c_i,2)=1$ and $\gcd(c_i, 3)=1$. It is clear that $\mathscr{Q}$  and $\mathscr{B} \subseteq \mathcal{M}_{\mathscr{Q}}$ are primitive. Observe that $\mathscr{Q}$ does not satisfy the condition (\ref{eq:e1}) as $10 \neq \gcd (b, \lcm(\mathscr{Q}))$  for any $b \in \mathscr{B}$. We take $\mathscr{Q}'=\{4, 6\}$ then the condition (\ref{eq:e1}) is satisfied. In fact, if we take $\mathscr{Q}'= \{4, 18\}$, then also (\ref{eq:e1}) is satisfied. Hence, $\mathscr{Q}'$ is not unique. We proceed by taking $\mathscr{Q} = \{4, 6\}$. So, $\lcm(\mathscr{Q}) = 12$. Furthermore, by (\ref{eq:e2}), $K=3$ and thus $M=48$. Hence, $\ell=144$ and for any $\varepsilon, \delta \in \{0, 1\}$, the exchangeable blocks are as follows:
    $$D_{\varepsilon, \delta} = \eta_{\mathscr{Q}} [1, 48]\varepsilon 0^{45}\delta 0 \eta_{\mathscr{Q}}[1, 48].$$ $$C = 10^{143}.$$
\end{example}

\section{On automorphisms of $\mathscr{B}$-admissible subshifts}\label{sec:admissible}

Let $N\in \N$. We say that a collection $(A_i)_{i\in I}$, where $I$ is a set and every $A_i$ is either a set of integers or a single integer, is {\em $N$-sparse}, if  $|a-a'|>N$ when  $a\in A_i, a'\in A_j$ and $i,j\in I$, $i\neq j$. If $A_i$ is a singleton, we often identify $A_i$ with its unique element. In particular, we  say that a set $A\subseteq\Z$ is {\em $N$-sparse} if $|a-a'|>N$ for every distinct $a,a'\in A$.

\begin{definition}\label{def:contour}
Given a primitive set $\mathscr{B}\subseteq\N\setminus\{1\}$ and a finite set $S\subseteq \mathscr{B}$ by an {\em $S$-contour} we mean a set of the form
$$
K_{\bf r}:=\bigcup_{b\in S}(r_b+b\Z)
$$
for some ${\bf r}=(r_b)_{b\in S}\in\Z^S$ satisfying the following conditions:
\begin{enumerate}
\item[(a)] the set $K_{\bf r}$ is minimal with respect to the inclusion among all sets of the form $\bigcup_{b\in S}(r'_b+b\Z)$, where $r'\in\Z^S$,
\item[(b)] the set $\Z\setminus K_{\bf r}$ is $\mathscr{B}$-admissible. \footnote{It follows by the condition (a) that $\Z\setminus K_{\bf r}\neq \emptyset$. Let us mention that it is an open problem (called the Schinzel's Conjecture in \cite{Zhi}) if $\bigcup_{b\in S}(r_b+b\Z)\neq \Z$ for every collection $r=(r_b)_{b\in S}\in \Z^{S}$, provided  the set $S$ is primitive. It is shown in  \cite[Theorem 2]{Schinzel} that the assertion follows from the Erd\"os-Selfridge Conjecture (if $\Z$ is a union of finite family of arithmetic progressions $r_b+b\Z$, $b\in S\subseteq\N\setminus\{1\}$, then one of the numbers $b$ is even).}
\end{enumerate}
\end{definition}

\begin{remark}\label{rem:crit} Observe that if $S\subseteq S'\subseteq \mathscr{B}$ and there is an $S$-contour, then there is an $S'$-contour.
Indeed, if $K_{\bf r}$ is an $S$-contour, where ${\bf r}=(r_b)_{b\in S}$, then, as $\Z\setminus K_{\bf r}$ is $\mathscr{B}$-admissible, for every $b'\in S'\setminus S$ there exists $r_{b'}\in \Z$ such that $r_{b'}+b'\Z\subseteq K_{\bf r}$. Let ${\bf r'}=(r_b)_{b\in S'}$. Then $K_{\bf r'}=K_{\bf r}$ and it is an $S'$-contour. It follows by  Theorem \ref{thm:main} that
if there exists a filtration $S_1\subset S_2\subset \ldots$ of $\mathscr{B}$ by finite sets such that for every $i\in\N$ and every ${\bf r}\in\Z^{S_i}$ the set $\Z\setminus K_{\bf r}$ is not $\mathscr{B}$-admissible then $\A_{\mathscr{B}}$ admits only trivial automorphisms.
\end{remark}
\begin{remark}\label{rem:sparse}
Assume that $F$ is a continuous endomorphism of the full shift $(\{0,1\}^{\Z},\sigma)$ defined by a code of radius $N$.
Assume moreover that  $(A_i)_{i\in I}$ is an $N$-sparse family of subsets of $\Z$. Then
$$
\supp(F(\raz_{\bigcup_{i\in I}A_i}))=\bigcup_{i\in I}\supp(F(\raz_{A_i})).
$$
In particular, if   $0\in\supp(F(\raz_{\{0\}}))$ and a set  $T\subseteq \Z$ is $N$-sparse, then $T\subseteq\supp(F(\raz_{T}))$.
\end{remark}

We prove the implication (a)$\Rightarrow$(c) of Theorem~\ref{thm:main} in Subsection \ref{sec:ac} and the implication (c)$\Rightarrow$(b) of Theorem~\ref{thm:main} in Subsection \ref{sec:cb}
Notice that (b)$\Rightarrow$(a) of Theorem~\ref{thm:main} is trivial.


\subsection{Proof of (a)$\Rightarrow$(c) in Theorem \ref{thm:main}}\label{sec:ac}
The idea is contained in  \cite{LRS} and   \cite{Me}.

\begin{lemma}\label{lem:repeat}
Let $A\subset \Z$ be a finite nonempty $\mathscr{B}$-admissible set and $M$  a fixed number. There exists an  infinite $M$-sparse sequence $(\ell_i)_{i\in\N\cup\{0\}}$ of integers such that the set
$$
A+\{\ell_i:i\in\N\cup\{0\}\}
$$
is $\mathscr{B}$-admissible, and for every $b\in\mathscr{B}$, $b|\ell_i$ for $i$ large enough.
\end{lemma}

\begin{proof}
For $n\in\N\cup\{0\}$ let $S_n=\{b\in\mathscr{B}:b\le n|A|\}$. Then $S_0=\emptyset$. We set $\ell_0=0$ and for $i>0$ we choose $\ell_i$ divisible by $\lcm(S_i)$ such  that $|\ell_i-\ell_j|>M$ provided $i\neq j$. Note that for every $b\in\mathscr{B}$, $b|\ell_i$ for $i$ large enough. By induction on $i$ we prove that $A+\{\ell_i:i\in\N\}$ is $b$-admissible for every $b\in S_i$.
Take $b\in S_{i+1}\setminus S_i$ for some $i\ge 0$. If $i=0$ then
$$
(A+\{\ell_i:i\in\N\})\mod b= A\mod b
$$ because $b|\ell_i$ for every $i>0$. Therefore $A+\{\ell_i:i\in\N\}$ is $b$-admissible. If $i>0$, then $b|\ell_j$ for $j\ge i+1$, hence
$$
(A+\{\ell_i:i\in\N\}) \mod b= (A+\{\ell_1,\ldots,\ell_i\})\mod b.
$$
Since $|A+\{\ell_1,\ldots,\ell_i\}|\le i|A|<b$, the set $A+\{\ell_i:i\in\N\}$ is $b$-admissible. The proof is complete.
\end{proof}

\begin{corollary}\label{cor:witness}
Let $A\subset \Z$ be a finite $\mathscr{B}$-admissible set and $M$  a fixed number. For every $b_0\in\mathscr{B}$ there exist integers $n_a$ for $a\in A$ such that the set
$A\cup\{n_a:a\in A\}$ is $\mathscr{B}$-admissible, $a=n_a\mod b_0$ for every $a\in A$ and the collection $A, (n_a)_{a\in A}$ is $M$-sparse.
\end{corollary}

\begin{proof}
For every $a\in A$ we choose $\ell_{i_a}$ from the numbers $\ell_i$ in Lemma \ref{lem:repeat} such that $b_0|\ell_{i_a}$,  $\ell_{i_a}>\max\{|a|:a\in A\}+M$ and $\ell_{i_a}\neq \ell_{i_{a'}}$ whenever $a\neq a'$. We set $n_a=a+\ell_{i_a}$ for $a\in A$.
\end{proof}

\begin{lemma}\label{lem:kontur}
Let $F$ be an automorphism of $\A_{\mathscr{B}}$ given by a code
$$
 f:\{0,1\}^{[-N,N]}\rightarrow\{0,1\}
$$
for some $N\in\N$. Assume that $f(0^N10^N)=1$ and $f(B)=1$ for a $\mathscr{B}$-admissible block $B\in\{0,1\}^{[-N,N]}$, $B\neq 0^{2N+1}$ such that $0\notin\supp(B):=\{i\in [-N,N]: B[i]=1\}$ {and $S$ is a finite subset of $\mathscr{B}$ such that $[1,N]\cap\mathscr{B}\subseteq S$}.
Then for every collection of integers ${\bf r}=(r_b)_{b\in S}\in \Z^S$ such that  the set
$$
K_{\bf r}:=\bigcup_{b\in S}(r_b+b\Z)
$$
is disjoint with $\supp(B)$,
the set
$\Z\setminus K_{\bf r}$
is $\mathscr{B}$-admissible.
In particular, there exists an $S$-contour disjoint with $\supp(B)$.
\end{lemma}

\begin{proof}Let $A=\supp(B)$. As $A$ is $\mathscr{B}$-admissible, there exist integers $r_b$ for $b\in S$ such that \[A\subseteq \Z\setminus\bigcup_{b\in S}(r_b+b\Z).\]
Suppose that the set $\Z\setminus\bigcup_{b\in S}(r_b+b\Z)$  is not $\mathscr{B}$-admissible. Then there exists a finite set $L\subseteq \Z\setminus\bigcup_{b\in S}(r_b+b\Z)$ such that $A\cup L$ is not $\mathscr{B}$-admissible.

Let $L$ be a minimal set with that property. Clearly, $A\cup L$ is $S$-admissible, so
 there exists $b_0\in\mathscr{B}\setminus S$ such that $A\cup L$ is not $b_0$-admissible. Let \[A'=(A\cup L)\setminus b_0\Z.\] Then $A'$ misses the residue class of 0 modulo $b_0$, so $A'$ is a proper subset of $A\cup L$. Since $b_0>N$ by the choice of $S$  and $A\subseteq [-N,N]$, we get $A\subset A'$, so $A'=A\cup(L\setminus b_0\Z)$ and $L\setminus b_0\Z$ is a proper subset of $L$. By the minimal choice of $L$, $A'$ is $\mathscr{B}$-admissible.

 Moreover, every residue class modulo $b_0$ different than 0 is represented in $A'$.  By applying Corollary \ref{cor:witness} to $A=A'$ and $M=N$ we prove that
 there exist integers $n_a$ for $a\in A'$ such that the set
\[A'\cup\{n_a:a\in A'\}\] is $\mathscr{B}$-admissible, $a=n_a\mod b_0$ for every $a\in A'$  and the collection $A', (n_a)_{a\in A'}$ is $N$-sparse. It follows that
every residue class modulo $b_0$ different than 0 is represented in $\{n_a:a\in A'\}$.
Let \[x=\raz_{A\cup\{n_a:a\in A'\}}.\] Then, by Remark \ref{rem:sparse}, the support of $F(x)$ contains $0$ and $n_a$ for $a\in A'$, so it is not $b_0$-admissible, a contradiction.

One can  choose $r_b$ such that the set $K_{\bf r}=\bigcup_{b\in S}(r_b+b\Z)$ is minimal in the family of the sets of the form $\bigcup_{b\in S}(s_b+b\Z)$, that are disjoint with $A$. As $A$ is $\mathscr{B}$-admissible, the family is nonempty. Therefore the last statement of the lemma follows from the remaining part.

\end{proof}


Now we prove the implication (a)$\Rightarrow$(c) of Theorem~\ref{thm:main}. Assume that the condition (c) is not satisfied and
 $F$ is an automorphism of $\A_{\mathscr{B}}$ given by a code
$$
 f:\{0,1\}^{[-N,N]}\rightarrow\{0,1\}
$$
for some $N\in\N$.
First observe that
\begin{equation}\label{eq:aut1}
f(0^{2N+1})=0,
\end{equation}
where $0^{2N+1}=00\ldots 0$. Suppose otherwise. Then $F(0^{\Z})=1^{\Z}\notin \A_{\mathscr{B}}$, a contradiction.

Next we show that
\begin{equation}\label{eq:aut2}
f(e_k)=1\;\text{for some}\;k\in[-N,N],
\end{equation}
where $e_k\in\{0,1\}^{[-N,N]}$ is the block of length $2N+1$  having 1 on the $k$th position and zeros elsewhere.
Indeed, otherwise $F(\raz_{\{0\}})=0^{\Z}$, which contradicts the assumption that $F$ is 1-1.

After composing $F$ with $\sigma^{-k}$ (and modifying $N$, if necessary), we can assume that
$$f(e_0)=1.$$
Under this assumption we will prove that $F$ is the identity map.

We show that
\begin{equation}\label{eq:aut3}
 B\in\{0,1\}^{[-N,N]}, 0\notin \supp(B) \Rightarrow f(B)=0,
\end{equation}
that is, the code $f$ is monotone. Suppose for the contradiction that $f(B)=1$ for some block, whose support $A:=\supp(B)$ does not contain $0$. By Lemma \ref{lem:kontur}
there exists a
finite set $S\subseteq \mathscr{B}$ and an $S$-contour $K_{\bf r}$, contrary to our assumption that (c) does not hold.

It follows by (\ref{eq:aut3}) that $F(x)\preccurlyeq x$ for every $x\in \A_{\mathscr{B}}$. Thus $F(\raz_{\{0\}})=\raz_{\{0\}}$ and by applying the above arguments to $F^{-1}$ we get that
$F^{-1}(y)\preccurlyeq y$ for every $y\in \A_{\mathscr{B}}$ and we conclude that $F(x)=x$ for every $x\in \A_{\mathscr{B}}$. The proof of the implication (a)$\Rightarrow$(c) is complete.

\begin{remark} Under the hypothesis that the condition (c) is not satisfied one can prove that there is no nontrivial injective continuous map $F:\A_{\mathscr{B}}\rightarrow \A_{\mathscr{B}}$ commuting with the shift. Indeed,
it follows by (\ref{eq:aut3}) that, after composing with a power of the shift, if necessary,  for every $x,x'\in \A_{\mathscr{B}}$:
\begin{equation}\label{eq:aut4}
x\preccurlyeq x'\Rightarrow F(x)\preccurlyeq x'.
\end{equation}
In particular, $F(x)\preccurlyeq x$ for every $x\in \A_{\mathscr{B}}$.
Now, by the arguments form \cite{LRS} (see the proof of Lemma 4.2 there) we see that
if $F(x)\prec x$ for some $x\in \A_{\mathscr{B}}$, then (as $F$ is 1-1) we have an infinite chain
$$
\ldots \prec F^k(x) \prec \ldots \prec F^2(x) \prec F(x) \prec x,
$$
which is impossible if the support of $x$ is finite. Therefore, $F(x)=x$ for every $x$ with the finite support and hence for every $x\in \A_{\mathscr{B}}$, since the elements with finite support form a dense subset of  $\A_{\mathscr{B}}$.
\end{remark}

\begin{example} Let $\mathscr{B}=\{6,10,15q_1,15q_2,\ldots\}$, where $q_1,q_2,\ldots$ is an infinite sequence of pairwise coprime numbers that are coprime to $2,3,5$. Let $S=\{6,10\}$. It is elementary that the set $6\Z\cup (5+10\Z)$ is minimal among the sets of the form $(r+6\Z)\cup (s+10\Z)$.  Since $15\Z\subseteq 6\Z\cup (5+10\Z)$, the set $\Z\setminus (6\Z\cup (5+10\Z))$ is $\mathscr{B}$-admissible. The set $6\Z\cup (5+10\Z)$ satisfies the minimality condition, so it is an $S$-contour.
\end{example}

\begin{example}\label{ex:78} It may happen that a set $\mathscr{B}$  admits no $S$-contour for any finite subset $S$, but the set $\Z\setminus K_{{\bf r}}$ is $\mathscr{B}$-admissible for some ${\bf r}\in \Z^S$ and a finite set $S\subset\mathscr{B}$. Let
$$
\mathscr{B}=\{30, 66, 78, a=1430=2\cdot 5\cdot 11\cdot 13, cq_1, cq_2,\ldots\},
$$
where $c=5\cdot 11\cdot 13$ and $q_1,q_2,\ldots$ is an infinite sequence of pairwise coprime numbers that are coprime to $2,3,5,11,13$ and  $\sum_{i\ge 1}\frac{1}{q_i}<1$. Let $S=\{30, 66, 78, 1430\}$.

One checks that
\begin{equation}\label{eq:nonsyn1}
a\Z\subseteq K_{{\bf r}'}:=30\Z\cup (22+66\Z)\cup (26+78\Z),
\end{equation}
 where $r'=(r'_{30}=0,r'_{66}=22,r'_{78}=26, r'_a=0)$. Since  $K_{\bf r'}$ consists of even numbers and $cq_i$ are odd, the set $\Z\setminus K_{\bf r'}$ is not $\mathscr{B}$-admissible. One checks that
\begin{equation}\label{eq:nonsyn2}
cq_i\Z\subset c\Z\subset K_{\bf r}:=30\Z\cup (22+66\Z)\cup (26+78\Z)\cup (c+a\Z), \end{equation}
for every $i\in\N$. Here ${\bf r}=(r_{30}=0,r_{66}=22,r_{78}=26, r_a=c)$, so the set $\Z\setminus K_{\bf r}$ is $\mathscr{B}$-admissible, but $K_{\bf r}$ is not an $S$-contour, as $K_{\bf r'}$ is a proper subset of $K_{\bf r}$.

 We claim that there is no $S$-contour for any finite set $S$. Assume that $S$ is a finite subset of $\mathscr{B}$ and $K_{\bf r}$, where ${\bf r}\in\Z^S$,  is minimal. Assume moreover that  $\Z\setminus K_{\bf r}$ is $\mathscr{B}$-admissible. By Remark~\ref{rem:crit}, without loss of generality we can assume that $30,66,78, a\in S$. Thanks to the assumption  $\sum_{i\ge 1} 1/q_i<1$ we have
$$
r_{b_0}+b_0\Z\nsubseteqq \bigcup_{b\in S\setminus\{b_0\}}(r_b+b\Z)
$$
if $b_0\in\{30, 66, 78\}$. Indeed, suppose that the inclusion holds. Then
\begin{equation*}
\begin{split}
\frac{1}{b_0}&=d(r_{b_0}+b_0\Z)=d\left(\bigcup_{b\in S\setminus\{b_0\}}(r_b+b\Z)\cap(r_{b_0}+b_0\Z)\right)\leq\sum_{b\in S\setminus\{b_0\}}d((r_b+b\Z)\cap(r_{b_0}+b_0\Z))\\
&\leq\sum_{b\in S\setminus\{b_0\}}\frac{1}{\lcm(b,b_0)}=\sum_{b\in S\setminus\{b_0\}}\frac{\gcd(b,b_0)}{b\cdot b_0}.
\end{split}
\end{equation*}
So we have
\[1\leq\sum_{b\in S\setminus\{b_0\}}\frac{\gcd(b,b_0)}{b}.\]
But observe that if $b_0=78$ then the r.h.s. of the inequality is less than or equal to
$$
\frac{\gcd(30,78)}{30}+\frac{\gcd(66,78)}{66}+\frac{\gcd(a,78)}{a}+\sum_{i\ge 1}\frac{\gcd(cq_i,78)}{cq_i}=\frac15+\frac{1}{11}+\frac{1}{5\cdot 11}+\frac{1}{5\cdot 11}\sum_{i\ge 1}\frac{1}{q_i}<1.
$$
If $b_0=30$ or $b_0=66$, this upper bound is still smaller, so we obtain a contradiction.

 As $\Z\setminus K_{\bf r}$ is $\mathscr{B}$-admissible, an arithmetic progression with the difference $cq_{i}$ is contained in $K_{\bf r}$ for every $i$, in particular for some $i_0$ large enough to satisfy $\gcd(q_{i_0},\lcm(S))=1$. Say, $s+cq_{i_0}\Z\subseteq  K_{\bf r}$ for such $i_0$ and some $s\in \Z$. Since $K_{\bf r}$ is $\lcm(S)$-periodic, it follows that
$$
s+c\Z=s+cq_{i_0}\Z+\lcm(S)\Z\subseteq K_{\bf r}.
$$
Again using the assumption $\sum_{i\ge 1}1/q_i<1$ we show that $s+c\Z$ is not contained in any finite union of progressions  $r_{cq_j}+cq_j\Z$.
It follows that  some progression with the difference $cq$, where $\gcd(q,\lcm(S))=1$, is contained in
$$
(r_{30}+30\Z)\cup  (r_{66}+66\Z)\cup (r_{78}+78\Z)\cup (r_{a}+a\Z).
$$
As before, by the $\lcm(S)$-periodicity of the latter set, we conclude that it contains a progression with the difference $c$:
 $$
 s'+c\Z\subseteq (r_{30}+30\Z)\cup  (r_{66}+66\Z)\cup (r_{78}+78\Z)\cup (r_{a}+a\Z)
$$
for some $s'\in\Z$. Since $s'+c\Z=(s'+a\Z)\cup (s'+c+a\Z)$ and one of the progressions $s'+a\Z, s'+c+a\Z$ is disjoint with $r_{a}+a\Z$ (since $c$ is odd), then
$(r_{30}+30\Z)\cup  (r_{66}+66\Z)\cup (r_{78}+78\Z)$ contains a progression of the form $s''+a\Z$ (with $s''=s'$ or $s''=s'+c$). One checks that this is possible only when the numbers $r_{30}, r_{66}, r_{78}, s''$ have the same parity. The minimality of $K_{\bf r}$ implies that $(r_{30}+30\Z)\cup  (r_{66}+66\Z)\cup (r_{78}+78\Z)$ contains $r_a+a\Z$ (since it contains $s''+a\Z$).  But then, as the numbers $cq_i$ are odd,  we observe that $\Z\setminus K_{\bf r}$ is not $\mathscr{B}$-admissible, as  $r_{30}, r_{66}, r_{78}$ have the same parity.
\end{example}




\subsection{When $\A_{\mathscr{B}}$ has only trivial automorphisms}\label{sec:onlytrivial}

We  give some necessary conditions for $\A_{\mathscr{B}}$ to have trivial automorphism group.

We say that a set $\mathscr{B}\subseteq\N\setminus\{1\}$ is {\em spiky} if every finite subset $S$ of $\mathscr{B}$ is saturated in $\mathscr{B}$, that is, if $S$ is a finite subset of $\mathscr{B}$ and  $b|\lcm(S)$ for some $b\in\mathscr{B}$, then $b\in S$. Clearly, if $\mathscr{B}$ is coprime, then it is spiky. Moreover, every spiky set is primitive.

\begin{lemma}\label{lem:incl} Assume that $\mathscr{B}$ is spiky.
Let $S\subset \mathscr{B}$ be finite and $b_*\in\mathscr{B}$. If
$$
s+b_*\Z\subseteq \bigcup_{b\in S}(r_b+b\Z)
$$
for some $s,r_b\in\Z$, then $b_*\in S$.
\end{lemma}

\begin{proof}
Induction on $|S|$. The statement is clear when $|S|=1$, since $
s+b_*\Z\subseteq r_b+b\Z$ implies $b|b_*$. Assume that $k>0$ and the lemma is true for sets $S$ of cardinality less than  $k$ and suppose that $|S|=k$. Choose arbitrary $b_0\in S$ and let $S'=S\setminus\{b_0\}$.
Then either $b_*\in S'$ or
$$
s'+L\Z\subseteq (s+b_*\Z)\setminus \bigcup_{b\in S'}(r_b+b\Z),
$$
for some $s'\in\Z$, where $L=\lcm(S'\cup\{b_*\})$, because,  if $b_*\notin S'$, the r.h.s set is nonempty (by the inductive hypothesis) and  $L$-periodic. Thus, if
$$
s+b_*\Z\subseteq \bigcup_{b\in S}(r_b+b\Z)
$$
and $b_*\notin S'$, then
$$s'+L\Z\subseteq \bigcup_{b\in S}(r_b+b\Z)\setminus \bigcup_{b\in S'}(r_b+b\Z)\subseteq r_{b_0}+b_0\Z,$$ hence $b_0|L$. Thus $b_0\in S'\cup\{b_*\}$, as $\mathscr{B}$ is spiky. Since $b_0\notin S'$, it follows that $b_0=b_*$, hence $b_*\in S$.
\end{proof}

\begin{Cor}\label{cor:spiky}
If $\mathscr{B}$ is spiky and infinite, then $\A_{\mathscr{B}}$ admits only trivial automorphisms.
\end{Cor}

\begin{proof}
We shall prove that there is no $S$-contour for any finite set $S\subseteq\mathscr{B}$. Assume for the contrary that $K_{\bf r}$ is an $S$-contour for some finite $S$ and ${\bf r}\in\Z^S$. Take $b\in\mathscr{B}\setminus S$.  As $\Z\setminus K_{\bf r}$ is $\mathscr{B}$-admissible, it misses at least one residue class modulo $b$, so $s+b\Z\subseteq K_{\bf r}$ for some $s\in Z$, which is impossible by Lemma \ref{lem:incl}.
\end{proof}

\begin{Cor}\label{cor:coprime}
If $\mathscr{B}$ contains an infinite coprime set, then $\A_{\mathscr{B}}$ admits only trivial automorphisms.
\end{Cor}

\begin{proof}
 We shall prove that there is no $S$-contour for any finite set $S\subseteq\mathscr{B}$. Assume for the contrary that $K_{\bf r}$ is an $S$-contour for some finite $S$ and ${\bf r}\in\Z^S$ and take $b\in\mathscr{B}$ that is coprime to $\lcm(S)$.   As $\Z\setminus K_{\bf r}$ is $\mathscr{B}$-admissible, it misses at least one residue class modulo $b$, so $s+b\Z\subseteq K_{\bf r}$ for some $s\in \Z$. But this is impossible, since the nonempty set $\Z\setminus K_{\bf r}$ is a union of arithmetic progressions with the difference $\lcm(S)$. It follows that $(s+b\Z)\cap (\Z\setminus K_{\bf r})\neq\emptyset$ for every $s\in\Z$.
\end{proof}

\begin{remark}
\begin{enumerate}
\item[(a)] A class of examples of minimal $\mathscr{B}$-free systems $X_{\eta_{\mathscr{B}}}$ admitting nontrivial automorphisms is given in \cite{DKK}, Section 3.4. They are defined by the sets $\mathscr{B}=\mathscr{B}_1^N$ of the form
$$
\{2^kc_k:k\in\N\}\cup\{2^{k-1}c_k^2:k<N\},
$$
where $c_1,\ldots,c_k,\ldots$ are pairwise coprime and $N\in\N\cup\{\infty\}$. It is shown in \cite[Proposition 3.23]{DKK} that if $N\ge 1$, then there are non-trivial automorphisms of $X_{\eta_{\mathscr{B}_1^N}}$.
By the results of \cite{Dymek} $X_{\eta_{\mathscr{B}_1^1}}$ admits only trivial automorphisms. The set $\mathscr{B}_1^1$ is spiky, so $\A_{\mathscr{B}_1^1}$ admits only trivial automorphism.
On the other hand,
the automorphism group of the hereditary closure of  $X_{\eta_{\mathscr{B}_1^N}}$ contains the automorphism group of the full shift as a subgroup by Corollary \ref{thm:main1}.
\item[(b)] The set $\mathscr{B}$ from Example \ref{ex:78} is not spiky, does not contain an infinite coprime set and the $\mathscr{B}$-admissible subshift admits no non-trivial automorphisms, as there is no $S$ contour for any finite subset $S$ of $\mathscr{B}$.
\item[(c)] If the set  $\mathscr{B}$ is finite, then the automorphism groups of  both $\widetilde{X_{\eta_{\mathscr{B}}}}$ and $\A_{\mathscr{B}}$  contain a subgroup isomorphic to the group of automorphisms of the full shift by Theorem  \ref{thm:main2} and Theorem \ref{thm:main},  whereas  $X_{\eta_{\mathscr{B}}}$ admits no non-trivial automorphisms.
\end{enumerate}
\end{remark}

\subsection{Proof of the implication (c)$\Rightarrow$(b) in Theorem \ref{thm:main}}\label{sec:cb}

The implication (c)$\Rightarrow$(b) of Theorem \ref{thm:main} follows immediately from Proposition \ref{prop:hereditaryperiodic} and the following proposition.

\begin{proposition} \label{prop:whencontour}
Assume that $S$ is a finite subset of $\mathscr{B}$ and $K_{{\bf r}}=\bigcup_{b\in S}(r_b+b\Z)$ is an $S$-contour.
There exist blocks $D_{0,0},D_{0,1},D_{1,0},D_{1,1},C$ satisfying the conditions (a), (b), (c) and (d) in Proposition \ref{prop:hereditaryperiodic} with $Y=\A_{\mathscr{B}}$ for some $\ell$.
\end{proposition}
\begin{proof}
Let $n$ be an integer such that $n+1\notin K_{{\bf r}}$, so $\eta_{{\bf r}}[n+1]=1$,
where
\begin{equation}\label{eq:etar}
\eta_{{\bf r}}:=\raz_{\Z\setminus K_{{\bf r}}}.
\end{equation}

Assume first that the block $\eta_{{\bf r}}[1,\lcm(S)]$ has at least two nonzero terms. Let $a\in\{2,3,\ldots,\lcm(S)\}$ be such that $\eta_{{\bf r}}[n+a]=1$.
Notice that $\eta_{{\bf r}}$ is $\lcm(S)$-periodic.
Put $k:=\lcm(S)$, $\ell:=3k$ and
\[B_{\bf r}:=\eta_{{\bf r}}[n+1,n+k].\]
In that case set

$$
\begin{array}{l}
D_{\varepsilon,\delta}:=B_{\bf r}\varepsilon\underbrace{0\ldots 0}_{a-2}\delta \underbrace{0\ldots 0}_{k-a}B_{\bf r},\\
C=1\underbrace{0\ldots 0}_{\ell-1}.
\end{array}
$$

If the block $\eta_{{\bf r}}[1,k]$ has exactly one nonzero term we modify the above definitions by setting $B_{\bf r}=\eta_{\bf r}[n+1,n+2k]$ and $k=2\lcm(S)$.

 We verify that the blocks $D_{0,0},D_{0,1},D_{1,0},D_{1,1},C$ indicated above satisfy the  conditions (a), (b), (c) and (d) in Proposition \ref{prop:hereditaryperiodic}.

(a) is immediate due to the fact that $D_{\varepsilon, \delta} [k+1] = \varepsilon, D_{\varepsilon, \delta}[k+a] = \delta, D_{\varepsilon, \delta}[2k+1]=1$  and $C[2k+1] = 0$.

(b) Assume that $y\in\A_{\mathscr{B}}$ and $y[m+1,m+\ell]=D_{\varepsilon,\delta}$ for some $m\in\Z$ and $\varepsilon,\delta\in\{0,1\}$. Then
\begin{equation}\label{eq:yieta}
y[m+1,m+k]=\eta_{{\bf r}}[n+1,n+k].
 \end{equation}
 After replacing $y$ by $\sigma^{m-n}(y)$ we can assume that  $y[n+1,n+k]=\eta_{{\bf r}}[n+1,n+k]$ (that is, $m=n$). Let $\varepsilon',\delta'\in\{0,1\}$. We claim that $y'\in\A_{\mathscr{B}}$, where $y'$ is the sequence defined by $y'[n+1,n+\ell]=D_{\varepsilon',\delta'}$ and $y'[s]=y[s]$ for $s\in\Z\setminus [n+1,n+\ell]$. 
 Since $y$ is $\mathscr{B}$-admissible, in particular it is $S$-admissible, there exist $s_b\in\Z$ for $b\in S$ such that $s_b+b\Z$ is disjoint with the support of $y$ for every $b\in S$, so, by (\ref{eq:yieta}), $$(s_b+b\Z)\cap (\Z\setminus K_{\bf r})\cap [n+1,n+k]=\emptyset.$$
 That means that $K_{\bf s}\cap [n+1,n+k]\subseteq K_{\bf r})\cap [n+1,n+k]$, where ${\bf s}=(s_b)_{b\in S}$. As the sets $K_{\bf s}$, $K_{\bf r}$ are $k$-periodic, it follows that $K_{\bf s}\subseteq K_{\bf r}$. As $K_{\bf r}$ is an $S$-contour, it follows by the condition (a) in Definition~\ref{def:contour} that $K_{\bf s}= K_{\bf r}$, so the support of $y$ is contained in $\Z\setminus K_{\bf r}$. In other words, $y\preccurlyeq \eta_{\bf r}$ coordinatewise. It follows immediately that $y'\preccurlyeq \eta_{\bf r}$ and, since $\eta_{\bf r}\in\A_{\mathscr{B}}$ (the condition (b) of Definition \ref{def:contour}) and $\A_{\mathscr{B}}$ is hereditary, $y'\in\A_{\mathscr{B}}$. The claim follows.

 The property (c) is to be checked directly, whereas (d) is a consequence of $\eta_{\bf r}\in\A_{\mathscr{B}}$ and the fact that every sequence as in (d) is less than or equal (coordinatewise) to $\sigma^{m}(\eta_{\bf r})$ for some $m\in\Z$.
\end{proof}

\begin{remark}
There are many possible  choices of blocks $D_{0,0},D_{1,0},D_{0,1}, D_{1,1}, C$ which satisfy conditions of Proposition~\ref{prop:hereditaryperiodic} and can be used in the proof of Proposition \ref{prop:whencontour}.
For example, any choice such that
$$
B_{{\bf r}}\varepsilon \underbrace{0\ldots 0}_{a-2}\delta\underbrace{0\ldots 0}_{k-a}B_{{\bf r}}\preccurlyeq D_{\varepsilon,\delta}\preccurlyeq B_{{\bf r}}B_{{\bf r}}B_{{\bf r}}
$$
does the job.
\end{remark}
\begin{remark}\label{rem:sofic}
Assume that $
\eta_{\bf r}$ is as in (\ref{eq:etar}) for an $S$-contour $K_{\bf r}$.
One can prove that the set
$$
X_{{\bf r}}=\{z\in\{0,1\}^{\Z}:z\preccurlyeq \sigma^j(\eta_{\bf r})\;\text{for some}\;j\in\Z\}
$$
is sofic and it is equal to the closure of the open set
$$
U_{{\bf r}}=\{z\in\A_{\mathscr{B}}: z[j+1,j+\lcm(S)]=\eta_{\bf r}[1,\lcm(S)]\;\text{for some}\;j\in\Z\}.
$$
{Moreover, the block $\eta_{\bf r}[1,\lcm(S)]$ is a synchronizing word in $X_{{\bf r}}$ and also in the whole subshift $\A_{\mathscr{B}}$. }
As the automorphism group of every uncountable sofic subshift contains the automorphism group of the full shift \cite{Sal}, it follows that the assertions  (a), (b), (c) of Theorem \ref{thm:main} are equivalent to the fact that $\A_{\mathscr{B}}$ contains a nonempty open $\sigma$-invariant subset $U$ such that the closure of $U$ is sofic.  Note that the subshift $\widetilde{X_{\eta_{\mathscr{B}}}}$ is not sofic and admits no synchronizing words, provided the set $\mathscr{B}$ is infinite and taut, see Corollary \ref{not_sofic}. {Also the subshift $\A_{\mathscr{B}}$ is not sofic for infinite $\mathscr{B}$ (Proposition \ref{lem:whensofic}), although it may admit synchronizing words, as noted above.}
\end{remark}



\section{Sofic property of subshifts defined by sets of multiples}

For any subshift $X$, we denote by $\mathcal{L}(X)$ its language, that is, the set of all finite blocks appearing in points of $X$. A block $C\in \mathcal{L}(X)$ is a {\em predecessor} (resp. {\em follower}) of $B\in\mathcal{L}(X)$ if  $CB\in \mathcal{L}(X)$ (resp. $BC\in\mathcal{L}(X)$). We denote the set of the predecessors (resp. followers) of a block $B$ by $Pred_X(B)$ (resp. $Foll_X(B)$), or just $Pred(B)$, $Foll(B)$ if the subshift $X$ is fixed. Weiss in~\cite{Weiss} proved that a subshift $X$ is sofic if and only if the range of the map \[\mathcal{L}(X)\ni B\mapsto Pred_X(B)\] is finite. The condition is equivalent to the finiteness of the range of the map \[\mathcal{L}(X)\ni B\mapsto Foll_X(B).\]
Recall that a block $B\in\mathcal{L}(X)$ is called \emph{synchronizing} if, for any $C,C'\in\mathcal{L}(X)$ such that $CB,BC'\in\mathcal{L}(X)$, one also has
\[
CBC'\in\mathcal{L}(X).
\]
Every sofic subshift admits synchronizing blocks. We show that in the subshifts containing the $\mathscr{B}$-free subshift and contained in its hereditary closure, there are no synchronizing words, see Proposition~\ref{taut_synch}. So these subshifts are not sofic, see Corollary~\ref{not_sofic}. Also the $\mathscr{B}$-admissible  subshift is not sofic, unless $\mathscr{B}$ is finite,  Proposition~\ref{lem:whensofic}. 


A specific example of hereditary subshift is so-called subordinate subshift which was introduced in \cite{MR3756340}. We say that a subshift $X\subset\{0,1\}^\Z$ is \emph{subordinate system under} $x\in X$ if we have \[X =\overline{\{\sigma^ny \colon y\preccurlyeq x,n\in\Z\}}.\]
 \begin{remark}\label{her_clos_subord}(\cite[Remark 5.4]{MR4938646})
     The hereditary closure of any $\mathscr{B}$-free subshift is a subordinate system under $\eta_{\mathscr{B}}$.
 \end{remark}
\begin{definition}
    For any $b\in\N$ a block $C$ is called \emph{$b$-full} if \[|\operatorname{supp} C \bmod b|=b-1.\]
\end{definition}
\begin{Th}({\cite[Corollary~4.32]{DKKL}})\label{eta_full}
    Let $\mathscr{B}$ be taut. Then for any $b\in\mathscr{B}$ we have
    \[|\operatorname{supp}\eta_{\mathscr{B}} \bmod b|=b-1.\]
\end{Th}
\begin{corollary}
    Let $\mathscr{B}$ be taut. For any $b\in\mathscr{B}$ there exists a $b$-full block in $\eta_\mathscr{B}$.
\end{corollary}
\begin{remark}\label{lem:b-full}
    Notice that for any $n<m$ we have $\supp\eta_{\mathscr{B}}[n,m]=[n,m]\cap\cf_\sB$. So
    \[\eta_{\mathscr{B}}[n,m]\text{ is }b\text{-full iff } \left(\forall_{k\in\Z} \ \cf_\sB\cap[n,m]\cap(b\Z+k)=\emptyset \implies b\mid k\right).\]
\end{remark}
\begin{Th}({\cite[Proposition~E]{DKKL}})\label{quasi_gen}
    For any $\mathscr{B}\subset\N$, ${\eta_{\mathscr{B}}}$ is a quasi-generic point for the Mirsky measure $\nu_{\eta_{\mathscr{B}}}$, i.e. for some $n_1<n_2<\ldots<n_k<\ldots$ we have
    \[\lim_{k\to\infty}\frac{1}{n_k}\sum_{n\leq n_k}\delta_{\sigma^n{\eta_{\mathscr{B}}}}=\nu_{\eta_{\mathscr{B}}}.\]
\end{Th}
\begin{Th}({\cite[Theorem 2]{MR3920387},\cite[Corollary 1.2]{MR4289651}})\label{taut_support}
   Let $\mathscr{B}\subset\N$. The set $\mathscr{B}$ is taut if and only if the subshift $X_{\eta_{\mathscr{B}}}$ is the topological support of the Mirsky measure, i.e. the measure $\nu_{\eta_{\mathscr{B}}}$ is positive for every open subset of $X_{\eta_{\mathscr{B}}}$.
\end{Th}
\begin{remark}
In more details, by Theorem 4.1 in \cite{DKKL}, ${\eta_{\mathscr{B}}}$ is quasi-generic
along any sequence $(\ell_i)$ realizing the lower density of $\mathcal{M}_{\mathscr{B}}$, i.e. $\lim_{i\to\infty} \frac{1}{\ell_i}|\mathcal{M}_{\mathscr{B}}\cap
[1, \ell_i]| = \liminf_{L\to\infty} \frac{1}{L}|\mathcal{M}_\mathscr{B}\cap[1, L]| =: \underline{d}(\mathcal{M}_\mathscr{B}$).
\end{remark}
\begin{corollary}\label{infinitely_many}
    If $\mathscr{B}$ is taut then any block in $\eta_{\mathscr{B}}$ appears infinitely many times in it (in both directions).
\end{corollary}
\begin{proof}
    By Theorem~\ref{taut_support}, any cylinder given by a block in ${\eta_{\mathscr{B}}}$ has positive Mirsky measure. By Theorem~\ref{quasi_gen}, any such block appears infinitely often in ${\eta_{\mathscr{B}}}$. Since ${\eta_{\mathscr{B}}}$ is symmetric, any block appears infinitely many times in both directions.
\end{proof}
\begin{proposition}\label{taut_synch}
    Let $\mathscr{B}\subset\N$ be infinite taut. Then there is no synchronizing block in $\mathcal{L}(Y)$ for any subshift $Y$ such that $X_{\eta_{\mathscr{B}}}\subseteq Y\subseteq\widetilde{X_{\eta_{\mathscr{B}}}}$.
\end{proposition}
\begin{proof}
    Suppose that $B\in\mathcal{L}(Y)$ is a synchronizing block. Since $Y\subseteq\widetilde{X_{\eta_{\mathscr{B}}}}$, by Remark~\ref{her_clos_subord}, there exists $n\in\Z$ such that
    \begin{equation}\label{block1}
    B\preccurlyeq{\eta_{\mathscr{B}}}[n+1,n+|B|].
    \end{equation} Since $\mathscr{B}$ is taut, by Corollary~\ref{infinitely_many}, any block in ${\eta_{\mathscr{B}}}$ appears infinitely often in it in both directions.
    So there exists $m\in\Z$ such that $m-n>|B|$ and
    \begin{equation}\label{block2}
    B\preccurlyeq\eta_{\mathscr{B}}[m+1,m+|B|].
    \end{equation}
    Put $T:=m-n$. Since $\mathscr{B}$ is infinite, there exists $b\in\mathscr{B}$ such that $b\nmid T$. By Theorem~\ref{eta_full} and Corollary~\ref{infinitely_many}, there exist $n',m'\in \Z$ such that $n'<n$ and $m'>m+|B|$ and ${\eta_{\mathscr{B}}}[n'+1,n]$ and ${\eta_{\mathscr{B}}}[m+|B|+1,m']$ are $b$-full. Put \[C:={\eta_{\mathscr{B}}}[n'+1,n]\text{ and }D:={\eta_{\mathscr{B}}}[m+|B|+1,m'].\] Then, by \eqref{block1}, \eqref{block2} and ${\eta_{\mathscr{B}}}\in Y$, we have \begin{equation}\label{eq:CBD}
    CB,BD\in\mathcal{L}(Y).
    \end{equation}Since $B$ is a synchronizing block, we get \begin{equation}\label{eq:CBBD}
    CBD\in\mathcal{L}(Y).
    \end{equation} Since $Y\subseteq\widetilde{X_{\eta_{\mathscr{B}}}}$, by Remark~\ref{her_clos_subord}, there exists $k\in\Z$ such that
    \begin{equation}\label{block3}
    CBD\preccurlyeq{\eta_{\mathscr{B}}}[k+1,k+|CBD|].
    \end{equation}
    Since $C,D$ are $b$-full, we obtain that ${\eta_{\mathscr{B}}}[k+1,k+|CBD|]$ is $b$-full too. Hence by Remark~\ref{lem:b-full}, we have
    \[\cf_\sB\cap [k+1,k+|CBD|]\cap(b\Z+s)=\emptyset\implies b\mid s.\]
    Since $C={\eta_{\mathscr{B}}}[n'+1,n]={\eta_{\mathscr{B}}}[k+1,k+|C|]$ and $D={\eta_{\mathscr{B}}}[m+|B|+1,m']={\eta_{\mathscr{B}}}[k+|CB|+1,k
+|CBD|]$, we have
    \begin{equation}\label{blockC}
    \cf_\sB\cap[n'+1,n]=\cf_\sB\cap[k+1,k+|C|]+n'-k\end{equation}
    and \begin{equation}\label{blockD}
    \cf_\sB\cap[m+|B|+1,m']=\cf_\sB\cap[k+|CB|+1,k+|CBD|]+m-k-|C|.\end{equation}
    Clearly, we have
    \[(\cf_\sB\cap[k+1,k+|C|]+n'-k)\cap(b\Z+n'-k)=\emptyset\]
    and
    \[(\cf_\sB\cap[k+|CB|+1,k+|CBD|]+m-k-|C|)\cap(b\Z+m-k-|C|)=\emptyset.\]
    So by \eqref{blockC} and \eqref{blockD}, we get
    \[\cf_\sB\cap[n'+1,n]\cap(b\Z+n'-k)=\emptyset\]
    and
    \[\cf_\sB\cap[m+|B|+1,m']\cap(b\Z+m-k-|C|)=\emptyset.\]
    Since $C$ and $D$ are $b$-full, by Remark~\ref{lem:b-full} we have
    \[b\mid n'-k\text{ and }b\mid m-k-|C|.\]
    Hence \[b\mid T=m-n=(m-k-|C|)-(n'-k).\] But this contradicts to the choice of $b$. So there is no synchronizing word in $\mathcal{L}(Y)$.
\end{proof}
\begin{corollary}\label{not_sofic}
For any infinite taut $\mathscr{B}$ subshifts $X_{\eta_{\mathscr{B}}}$ and $\widetilde{X_{\eta_{\mathscr{B}}}}$ are not sofic.
\end{corollary}
\begin{proof}
Any sofic subshift has a synchronizing block. But by Proposition~\ref{taut_synch}, neither $X_{\eta_{\mathscr{B}}}$ nor $\widetilde{X_{\eta_{\mathscr{B}}}}$ have such a block. The assertion follows.
\end{proof}
\begin{example}
Let $\mathscr{B}=\mathbb{P}$. Then $\mathscr{B}$ is not taut. Moreover, we have \[X_{\eta_{\mathscr{B}}}=\{\sigma^n{\eta_{\mathscr{B}}}: \ n\in\Z\}\cup\{\ldots0.000\ldots\}\]
and \[\widetilde{X_{\eta_{\mathscr{B}}}}=\{\sigma^n{\eta_{\mathscr{B}}}: \ n\in\Z\}\cup\{\sigma^n(\ldots00.100\ldots): \ n\in\Z\}\cup\{\ldots0.000\ldots\}.\]
Let $F\colon\widetilde{X_{\eta_{\mathscr{B}}}}\to\widetilde{X_{\eta_{\mathscr{B}}}}$ be a sliding block code given by the coding $f\colon\{0,1\}^{[-3,3]}\to\{0,1\}$:
\[f(B)=\begin{cases}
&1-B(0)  \text{ for } B\in\{0001000,0010000,0000100,0101000,0001010,0010100\}\\
&B(0) \text{ for } B\not\in\{0001000,0010000,0000100,0101000,0001010,0010100\}.
\end{cases}
\]
Then \[F^2=id.\]
Notice that $F$ cannot be extended to the automorphism of $\mathbb{A}_\mathscr{B}$. Indeed, suppose that $F$ can be extended to the automorphism of $\mathbb{A}_\mathscr{B}$. Consider $x=\ldots00.10000000001000\ldots\in\mathbb{A}_{\mathscr{B}}$. We have
\[F(x)=\ldots001.01000000010100\ldots.\]
But $\operatorname{supp}(F(x))=\{-1,1,9,11\}$ is not $3$-admissible. So $F(x)\not\in\mathbb{A}_{\mathscr{B}}$. The assertion follows.
\end{example}

\begin{Prop}\label{lem:whensofic}
The admissible shift $\mathbb{A}_{\mathscr{B}}$ is sofic if and only if the set $\mathscr{B}\subseteq \N\setminus\{1\}$ is finite.
\end{Prop}

\begin{proof}
The "if" part is easy. For the proof of the "only if" part  assume that the set $\mathscr{B}$ is infinite and let us enumerate the elements of the set $\mathscr{B}$:
$$
b_1<b_2<\ldots<b_n<\ldots.
$$
Observe that $b_k>k$ for every $k\in\N$ and $b_k+l-k\le b_l$ whenever $k\le l$. For $M\in\N\cup\{\infty\}$ and $n\in \N$, $n\le M$, let
$$
A_n^M=\{\lcm(b_1,\ldots,b_{m}):n\le m\le M\}.
$$
Clearly, $A_n^M\subseteq A_{n'}^{M'}$ whenever $n\ge n'$ and $M\le M'$. Every set $A_n^M$ is $\mathscr{B}$-admissible: it is enough to observe that $A_1^{\infty}$ is $\mathscr{B}$-admissible, {as
$$
A_1^{\infty}\mod b_k=\{b_1,\lcm(b_1,b_2),\ldots,\lcm(b_1,b_2,\ldots, b_{k-1}),0\}\mod b_k,
$$
so}
$|A_1^{\infty}\mod b_k|\le k <b_k$ for every $k\in\N$.

In what follows, in order to avoid complicated notation, we shall identify finite sets of integers with blocks in the following way: a finite set $A$ of natural numbers is identified with the block $\raz_A[1,\max(A)]$, whereas a finite set $C$ of negative numbers is identified with the block $\raz_C[\min(C),0]$.

 We shall prove that the range of the map $A\mapsto Pred_{\mathbb{A}_{\mathscr{B}}}$ is infinite, which in view of Theorem 4 in \cite{Weiss}, yields that $\mathbb{A}_{\mathscr{B}}$  is not sofic.

Since $\mathscr{B}$ is primitive, there is an infinite sequence $k_1<k_2<\ldots $ such that $b_{k_i}\nmid\lcm(b_1,\ldots,b_{k_i-1})$ for every $i\in \N$. {Indeed, otherwise the set of the primes that divide at least one element $b\in\mathscr{B}$ is finite, which is impossible when $\mathscr{B}$ is infinite and primitive.} Fix $N\in\N$, we shall prove that the range of the map $A\mapsto Pred(A)$ has cardinality at least $N$, more precisely, we prove that the sets $Pred_{\mathbb{A}_{\mathscr{B}}}(A_{k_1}^M),\ldots, Pred_{\mathbb{A}_{\mathscr{B}}}(A_{k_N}^M)$ are pairwise distinct, where  $M>k_N$ (the sets $A_{k_i}^M$ are identified with blocks according to the rule described above).
As $A^M_{k_{i+1}}\subset A^M_{k_{i}}$ for $i=1,\ldots, N-1$,
$$
Pred_{\mathbb{A}_{\mathscr{B}}}(A_{k_1}^M)\subset\ldots \subset Pred_{\mathbb{A}_{\mathscr{B}}}(A_{k_N}^M).
$$
We shall prove that the inclusions are strict.
Fix $1\le i\le N$. For every $s\in\{1,\ldots, b_{k_i}-1\}$, by Corollary~\ref{Cor_Dirichlet}, the set
\begin{equation}\label{eq:set}
s+b_{k_i}\Z\setminus (\bigcup_{j<k_i}(1+b_{j}\Z))
\end{equation}
is nonempty. 
For $s\in\{1,\ldots, b_{k_i}-1\}$ let us choose a negative integer $c_s$ belonging to the set (\ref{eq:set}) and let
$$C=\{c_s:s\in\{1,\ldots, b_{k_i}-1\},\;s\neq \lcm(b_{k_{i-1}},\ldots,b_{k_i-1})\mod b_{k_i}\}.$$
Observe that, as $b_{k_i}\nmid\lcm(b_1,\ldots,b_{k_i-1})$, we have  $b_{k_i}\nmid\lcm(b_{k_{i-1}},\ldots,b_{k_i-1})$, so $|C|=b_{k_i}-2$.
Every element of $A^M_{k_i}$ is divisible by $b_j$ for $j\le k_i$. It follows that the set $C\cup A^M_{k_i}$ misses the residue class of $1$ modulo $b_j$ for $j<k_i$ and it  misses the residue class of $\lcm(b_{k_{i-1}},\ldots,b_{k_i-1})$ modulo $b_{k_i}$. Finally, $ |A^M_{k_i}\mod b_m|\le m-k_i+1$, so $|(C\cup A^M_{k_i})\mod b_m|\le b_{k_i}-2+m-k_i+1<b_m$ for $m>k_i$. It follows that the set $C\cup A^M_{k_i}$ is $\mathscr{B}$-admissible, so $C\in Pred_{\A_{\mathscr{B}}}(A^M_{k_i})$.

On the other hand, $A^M_{k_{i-1}}$ meets the residue classes of $0$ and of $\lcm(b_{k_{i-1}},\ldots,b_{k_i-1})$ modulo $b_{k_i}$, so $C\cup A^M_{k_{i-1}}$ meets every residue class modulo $b_{k_i}$ hence $C\cup A^M_{k_{i-1}}$ is not $\mathscr{B}$-admissible and $C\notin Pred_{\mathbb{A}_{\mathscr{B}}}(A^M_{k_{i-1}})$.
\end{proof}

 \end{document}